\documentclass[11pt, twoside]{article}

\title{Fast relaxation of the random field Ising dynamics}
\author{Ahmed El Alaoui\thanks{Department of Statistics and Data Science, Cornell University.} \and Ronen Eldan\thanks{Microsoft Research and Weizmann Institute of Science.} \and Reza Gheissari\thanks{Department of Mathematics, Northwestern University.}  \and Arianna Piana\thanks{Department of Mathematics, Weizmann Institute of Science.}}

\usepackage{pgf,tikz}
\usepackage{mathrsfs}
\usepackage[backend=bibtex,bibstyle=numeric]{biblatex}
\usetikzlibrary{arrows}
\usepackage[a4paper,top=2.5cm,bottom=2.5cm,left=2.5cm,right=2.5cm]{geometry}

\usepackage{polynom}
\usepackage[colorinlistoftodos]{todonotes}
\usepackage{multicol}
\usepackage{mathtools}
\usepackage{amsthm}
\usepackage{amsmath}
\numberwithin{equation}{section}
\usepackage{graphicx}
\usepackage{subcaption}
\usepackage{amssymb}
\usepackage{dsfont}
\usepackage{systeme}
\usepackage{mdframed}
\usepackage{enumitem}
\usepackage{bbm}
\usepackage{bm}
\usepackage{setspace}
\usepackage[colorlinks=true, linkcolor=red, citecolor=cyan]{hyperref}
\theoremstyle{plain}
\usepackage{stmaryrd}
\theoremstyle{plain}
\newtheorem{Thm}{Theorem}[section]

\newtheorem{Lem}[Thm]{Lemma}
\newtheorem{Prop}[Thm]{Proposition}
\newtheorem{Cor}[Thm]{Corollary}

\newtheorem*{theorem*}{Theorem}

\theoremstyle{definition}
\newtheorem{Def}[Thm]{Definition}

\newtheorem{Ass}[Thm]{Assumption}
\newtheorem{Ex}[Thm]{Example}
\newmdtheoremenv{theo}{Theorem}
\newtheorem{lemma}[Thm]{Lemma}

\theoremstyle{remark}
\newtheorem{Rem}[Thm]{Remark}

\newtheorem{Obs}[Thm]{Observation}

\newcommand{\var}{\mathrm{Var}}

\newcommand{\cov}{\mathrm{Cov}}
\newcommand{\op}{\textup{op}}
\newcommand{\mix}{_\mathrm{mix}}

\newcommand{\dd}{\mathrm{d}}
\newcommand{\rmd}{\mathrm{d}}

\newcommand{\tr}{\mathrm{Tr}}

\newcommand{\EE}{{\mathbb{E}}}
\newcommand{\E}{{\mathbb{E}}}

\renewcommand{\P}{{\mathbb{P}}}
\newcommand{\R}{{\mathbb{R}}}
\newcommand{\ZZ}{{\mathbb{Z}}}

\def\In{\Lambda_n}

\newcommand{\cont}{\mathrm{cont}}

\newcommand{\real}{\mathbb{R}}

\newcommand{\one}{\mathbf{1}}
\newcommand{\eps}{\varepsilon}

\newcommand{\mE}{\mathcal{E}}
\def\sTV{\mbox{\rm\tiny TV}}
\newcommand{\salg}{\mbox{\rm\tiny alg}}
\newcommand{\RFIM}{\text{\rm\small RFIM}}
\newcommand{\SL}{\text{\rm\small SL}}
\newcommand{\WSM}{\text{\rm\small WSM}}
\newcommand{\SSM}{\text{\rm\small SSM}}
\newcommand{\FKG}{\text{\rm\small FKG}}
\newcommand{\osc}{\mathrm{osc}}
\newcommand{\bad}{\mathsf{Bad}}
\newcommand{\good}{\mathsf{Good}}
\newcommand{\influence}{{\mathsf{i}}}
\newcommand{\cC}{\ensuremath{\mathcal C}}
\newcommand{\cB}{\ensuremath{\mathcal B}}
\newcommand{\gap}{\text{\tt{gap}}}
\DeclareMathOperator{\sgn}{sgn}

\date{}
\begin{document}
\maketitle
\begin{abstract}
We study the convergence properties of Glauber dynamics for the random field Ising model (RFIM) with ferromagnetic interactions on finite domains of $\mathbb{Z}^d$, $d \ge 2$. 
Of particular interest is the \emph{Griffiths phase} where correlations decay exponentially fast in expectation over the quenched disorder, but there exist arbitrarily large islands of weak fields where low-temperature behavior is observed. 
Our results are twofold: 

1. Under \emph{weak spatial mixing}  (boundary-to-bulk exponential decay of correlations) in expectation, we show that the dynamics satisfy a \emph{weak} Poincar\'e inequality 
implying polynomial relaxation to equilibrium over timescales polynomial in the volume $N$ of the domain, and polynomial time mixing from a warm start. From this we construct a polynomial-time approximate sampling algorithm
based on running Glauber dynamics over an increasing sequence of approximations of the domain. 

2. Under \emph{strong spatial mixing} (exponential decay of correlations even near boundary pinnings) in expectation, we prove a \emph{full} Poincar\'e inequality, implying exponential relaxation to equilibrium and $N^{o(1)}$-mixing time.        
Both weak and strong spatial mixing hold at any temperature, provided the external fields are strong enough. 

Our proofs combine a stochastic localization technique which has the effect of increasing the variance of the field, with a field-dependent coarse graining which controls the resulting sub-critical percolation process of sites with weak fields.
\end{abstract}

\section{Introduction}

Given a finite graph $G = (V,E)$, the random field Ising model ($\RFIM$) at inverse-temperature $\beta>0$ on $G$ is the random probability measure on $\{-1,+1\}^V$ given by
 \begin{equation}\label{eq:mu_RFIM}
\mu_G(\sigma) \propto \, \exp\big\{  -  H(\sigma)\big\},~~~ \sigma \in \{-1,+1\}^V \, , 
\end{equation}
where $H$ is the random Hamiltonian  
\begin{align}\label{eq:Hamiltonian}
    H(\sigma) =   - \beta\sum_{(u,v) \in E} \sigma_u \sigma_v - \sum_{u \in V} h_u \sigma_u\,,
\end{align}
with $h=(h_u)_{u \in V}$ i.i.d.\ random variables with a symmetric distribution. 
The vector $h$ is referred to as the external field, and is viewed as a quenched source of disorder.   

The $\RFIM$ has seen much attention in the statistical physics literature on (subgraphs of) $\ZZ^d$ for $d\ge 2$. 
In $d=2$, Imry and Ma~\cite{Imry-Ma} famously predicted in 1975 that the presence of an arbitrarily small random field suppresses the low-temperature phase of the classical (zero-field $h = 0$) Ising model, so that at all values of $\beta$, the measure~\eqref{eq:mu_RFIM} exhibits uniqueness and decay of correlations. After much debate within the physics community, this prediction was confirmed by Aizenman and Wehr~\cite{Aizenman-Wehr} who showed that at all $\beta>0$ and for external fields with non-trivial distribution and of arbitrarily small but non-zero variance, the expected influence on the spin at the origin of boundary conditions at distance $r$ away must decay with $r$. 
A recent flurry of activity on this problem gave quantitative rates on this decay of influence with $r$~\cite{Chatterjee-RFIM,Aizenman-Peled,Aizenman-Harel-Peled}, and in the Gaussian case $h_u \sim N(0,\sigma^2)$ culminated in the breakthrough work~\cite{DingXia-RFIM} showing that the influence decays exponentially at all temperatures and $\beta>0$ and all field variances $\sigma^2>0$. 
In dimensions $d\ge 3$, the $\RFIM$ behaves very differently. If the field is sufficiently large as a function of $\beta$, influences decay exponentially fast in expectation over the field~\cite{Frohlich-Imbrie}. But unlike in $d=2$, when $d\ge 3$ if the disorder is relatively small, there is a phase transition as $\beta$ is varied. The lack of influence decay was shown in~\cite{Imbrie} at $\beta = \infty$ and at $\beta$ sufficiently large in \cite{Bricmont-Kupiainen-RFIM}, and more recently in the entire low temperature phase in~\cite{ding2022long}.

In this paper, we are interested in the behavior of the \emph{Glauber dynamics} for the $\RFIM$ on (subsets of) $\mathbb Z^d$ for general $d\ge 2$. The Glauber dynamics is the continuous-time Markov process $(X_t(v))_{v \in V, t\ge 0}$, reversible w.r.t.\ $\mu_G$, which assigns each vertex $v\in V$ a rate-1 Poisson clock, and when the clock at $v\in V$ rings at time $t$, it resamples $X_t(v)$ according to the conditional distribution 
\begin{equation}\label{eq:glauber}
X_t(v) \,\sim\, \mu_G\Big(\sigma_v \in \cdot \,\,\Big|\, (\sigma_w)_{w\ne v} = (X_{t^-}(w))_{w\ne v}\Big)\, .
\end{equation}
This process was introduced in 1963 in~\cite{Glauber} and has since seen immense attention both as a physical model for the thermodynamic evolution of an Ising system out of equilibrium, and as a Gibbs sampler MCMC algorithm for the measure $\mu_G$. The central quantity to study for the Glauber dynamics is its rate of convergence to $\mu_G$, typically measured either in $\ell^2$, via spectral gap or Poincar\'e inequalities, or total-variation distance, via mixing times. 

For the classical Ising model (zero-field), the convergence rate to equilibrium of the Glauber dynamics has been extensively studied over the last thirty years, and is by now fairly well understood. Namely, for the Ising Glauber dynamics on $\Lambda_n = [-n,n]^d \cap \mathbb Z^d$, weak and strong spatial mixing properties (exponential decay of correlations away from boundary conditions, and near boundary conditions, respectively) have been shown to imply order-1 inverse spectral gap, and exponentially fast relaxation to equilibrium~\cite{MaOl1,MaOl2}. These spatial mixing properties hold at all high temperatures $\beta<\beta_c(d)$ in all $d\ge 2$, where $\beta_c(d)$ is the critical inverse temperature of the zero-field Ising model on $\ZZ^d$. Implications from spatial mixing to temporal mixing have since been a central tool in the analysis of mixing times of spin systems on lattices: see e.g.,~\cite{Cesi,DSVW,LS-cutoff,GhSi-spatial-mixing-RC}. 
On the other hand, as soon as $\beta>\beta_c(d)$, the mixing time becomes exponentially slow in $n^{d-1}$ due to the bottleneck between having most spins take the value $+1$ and most spins take the value $-1$: we refer to the lecture notes~\cite{MartinelliLectureNotes} for references on mixing times of Ising Glauber dynamics on $\mathbb Z^d$.

In the presence of quenched disorder, as in the $\RFIM$, such implications going from exponential decay of correlations (now only assumed to hold in expectation over the disorder $h$--see Definition~\ref{ass:general-WSM} for the formal definition) to fast relaxation to equilibrium may no longer hold. This is because of what is known as \emph{Griffiths singularities}~\cite{GriffithsSingularity}, whereby if $\beta>\beta_c(d)$, there are small but positive density regions in $\Lambda_n$ where the field is very close to zero, and low-temperature behavior is observed, even though in expectation over the field, correlations decay exponentially. 
In fact, it can be shown that in this \emph{Griffiths phase}, the dynamics can slow down so that even at large field values, the spectral gap has to tend to zero at least super-logarithmically in $n$ if $\beta>\beta_c(d)$ (see e.g.,~\cite{CMM-Disordered-magnets} in the context of random interaction strengths rather than random fields, or our Section~\ref{sec:lower-bound}). 
In particular, this precludes straightforward applicability of the recently developed general purpose tools, such as spectral independence, to deduce fast mixing under deterministic influence or covariance matrix bounds; e.g.,~\cite{ALO,EntropicIndependence1,BauerschmidtBodineau,EKZ,chen2022localization}. 
This begs the question, most pertinent in the regime where exponential decay of correlations holds for the $\RFIM$ in expectation, but not for the corresponding zero-field model---e.g., $\beta>\beta_c(d)$ and large variance field---of how fast the Glauber dynamics relaxes to equilibrium.  

The $\RFIM$ dynamics are also interesting from the perspective of computational complexity of approximate sampling and counting: the task of approximately sampling from the zero-field Ising distribution was shown to be algorithmically tractable in polynomial time on general graphs in~\cite{Jerrum-Sinclair}. On the other hand, with arbitrary external fields, the sampling task is $\text{\rm\small\#BIS}$ hard~\cite{Ising-with-field-BIS-hard}, namely as hard as approximately counting the number of independent sets on a bipartite graph, a task expected to be intractable in polynomial time. From this perspective, the $\RFIM$ can be viewed as a proxy for the average-case hardness of the sampling task. Recently, it was shown in~\cite{HLPRW-RFIM-sampling} that on general graphs, if the variance of the fields is sufficiently large, there exists a polynomial-time approximation algorithm for the $\RFIM$ which is based on the self-avoiding walk tree of Weitz~\cite{Weitz}. The authors of \cite{HLPRW-RFIM-sampling} pose the question of whether Markov chain approaches can provide a similar polynomial-time sampling scheme.

In this paper, we prove Poincar\'e-type inequalities and bound the rate of convergence for the $\RFIM$ Glauber dynamics on boxes of side-length $n$ in $\mathbb Z^d$ whenever it exhibits exponential decay of correlations in expectation. The form of exponential decay falls into two classes, weak spatial mixing and strong spatial mixing: the distinction is that the latter requires the exponential decay of correlations in expectation to hold even in the presence of frozen nearby spins (i.e., nearby boundary conditions). 
Our main result is that under the weak spatial mixing assumption, the Glauber dynamics satisfies a \emph{weak Poincar\'e inequality}, and converges at least algebraically fast, i.e., inverse-polynomially in time, to equilibrium on polynomial in $n$ timescales: see Theorem~\ref{thm:algebraic_decay}. This also translates to a polynomial-time sampling scheme built from running Glauber dynamics chains on domains successively approximating $\Lambda_n$: see Theorem~\ref{thm:sampling_alg}. Under the stronger assumption of strong spatial mixing in expectation, which for instance holds at any $\beta$ if the field variances are sufficiently large as a function of $\beta$, we show a (full) Poincar\'e inequality with inverse spectral gap that is $n^{o(1)}$, and consequently, exponentially fast relaxation to equilibrium after such timescales: see Theorem~\ref{mainthm:SSM}. 

We note that as a consequence of the recently discovered  correlation inequality of~\cite{ding2022new}, strong spatial mixing holds \emph{uniformly} in the external field for $\beta<\beta_c(d)$, and hence one already obtains order-1 spectral gap for the $\RFIM$ Glauber dynamics by standard implications; see~\cite{MartinelliLectureNotes}. Our results are thus most pertinent in the Griffiths phase discussed above.   

\subsection{Main results}
In what follows, our domain of most interest is $\Lambda_n = [-n,n]^d \cap \mathbb Z^d$ endowed with nearest-neighbor adjacency. To state our results and execute our arguments, we will also have to consider the $\RFIM$ measure $\mu_{\Lambda}$ on general finite domains $\Lambda \subset \ZZ^d$. (By abuse of notation we let $\Lambda$ denote the subgraph of $\ZZ^d$ whose vertex set is $\Lambda$.)
We will always let $N = |\Lambda|$ be the volume of the domain under consideration, and we write $\P_h$ for the law of the random field $h$. 

For a vertex $o\in \ZZ^d$ and an integer $r\ge0$, we denote by $B_r(o)$ the ball centered at $o$ with radius $r$ in the $\ell_\infty$ distance in $\ZZ^d$.
For a subset $A \subseteq \Lambda$ we let $\sigma_A = (\sigma_u)_{u \in A}$. We use $\partial A$ to denote the exterior vertex boundary of $A$ (with respect to $\Lambda$). 
Given a `boundary condition' $\tau \in  \{-1,+1\}^{\partial A}$ we define 
\begin{equation}\label{eq:mu_pinned}
\mu_{A}^{\tau}(\sigma_A)  := \mu_{\Lambda}(\sigma_A \in \cdot \, |\,  \sigma_{\partial A} = \tau) \, ,
\end{equation}
as the distribution of $\sigma_A$ induced by the $\RFIM$ measure $\mu_{\Lambda}$ on $\{-1,+1\}^A$ conditional on $\sigma_{\partial A} = \tau$. 
 When $\tau$ is the all `$+$' or the all `$-$' configuration we simply write  $\mu_{A}^{+}$ and $\mu_{A}^{-}$ respectively.


\paragraph{Weak spatial mixing.} Our first assumption is of \emph{weak spatial mixing} ($\WSM$) type where in expectation, the influence of a boundary on the `root' spin decays exponentially in the distance from the boundary. This minimal level of exponential decay should hold throughout the uniqueness region of the $\RFIM$. 

\begin{Def}\label{ass:general-WSM}
    We say that the $\RFIM$ on $\Lambda \subset \ZZ^d$ satisfies $\WSM(C)$ for some constant $C>0$ if for all $o \in \Lambda$ and all $r\ge 1$ we have 
    \begin{align}\label{eq:expdecay}
       \EE_h \, \big[d_{\sTV} \big( \mu_{B_r(o)\cap\Lambda}^{+}(\sigma_o \in \,\cdot\,)\, ,\, \mu_{B_r(o)\cap \Lambda}^{-}(\sigma_o \in \,\cdot\,) \big)\big] \leq C e^{- r/C}  \, ,
    \end{align}
    where $d_{\sTV}(\cdot,\cdot)$ is the total-variation distance, and the expectation is over the field $h$. Importantly, all boundary pinnings in $\WSM$ are at distance at least $r$ from $o$ (c.f.\ Definition~\ref{assump:SSM-RFIM0} of $\SSM$). 
\end{Def}

Denote by $(P_t)_{t \ge 0}$ the Markov semigroup for the continuous-time Glauber dynamics~\eqref{eq:glauber}: i.e., for a given $h$ realization, $P_t\varphi(\sigma) = \EE[ \varphi(X_t) | X_0=\sigma]$ for tests function $\varphi: \{-1,+1\}^{V} \to \real$. For a probability measure $\mu$, let $\var_{\mu}(\varphi)=\mu(\varphi^2) - \mu(\varphi)^2$ be the variance of $\varphi$ under $\mu$, and let $\osc(\varphi) = \max \varphi - \min \varphi$ be the (square of the) oscillation of $\varphi$.

Under Definition~\ref{ass:general-WSM}, we show algebraic convergence to equilibrium for any test function, after polynomial in $|\Lambda|$ time, and the existence of a polynomial time sampling algorithm.  

\begin{Thm}\label{thm:algebraic_decay}
Suppose the $\RFIM$ on $\Lambda_n$ satisfies $\WSM(C)$ for some $C>0$. Then there exists $K, \alpha, \kappa>0$ depending on $d, \beta, C$ such that for all $\delta >0$ the following holds with $\P_h$-probability $1-\delta$: 
For all test functions $\varphi:  \{-1,+1\}^{\In} \rightarrow \R$ and all $t > 0$, 
 \begin{equation}\label{eq:algebraic}
        \var_{\mu_{\In}}(P_t \varphi) \le K n^{\kappa} \, \osc(\varphi)^2 \big/ (\delta t^{\alpha})\,,
    \end{equation}
i.e., the continuous-time Glauber dynamics for the $\RFIM$ on $\Lambda_n$ has algebraic convergence to equilibrium on $n^{\kappa/\alpha}$ timescales.   
\end{Thm}

\begin{Rem}
Our techniques can also be used to extend Theorem~\ref{thm:algebraic_decay} to the case of the torus $T_n = (\mathbb Z/n\mathbb Z)^d$. There, the vertex transitivity reduces $\WSM(C)$ to the expected difference between the marginals on the origin in boxes of radius $r$ with `$+$' and `$-$' boundary conditions.
\end{Rem}

The above algebraic convergence to equilibrium is a consequence of a \emph{weak Poincar\'e inequality}~\cite{rockner2001weak} and accompanying large set expansion. These notions have received attention since the classical work of Liggett~\cite{liggett1991l_2} on critical interacting particle systems, see e.g.,~\cite{Andrieu-weakPI-AoS,andrieu2022poincar,Grothaus-AoP-weakPI}. 
The bound of~\eqref{eq:algebraic} is weaker than what would be implied by a spectral gap lower bound in two ways: (1) the decay in time is inverse-polynomial (algebraic) instead of exponential, and (2) the right-hand side of the inequality features the oscillation of the test function rather than its variance w.r.t\ $\mu_{\In}$. 
A slightly stronger result which does not degenerate at $t=0$ and recovers the trivial inequality $ \var_{\mu_{\In}}(P_t \varphi) \le  \var_{\mu_{\In}}(\varphi)$ can be found in Section~\ref{sec:proof_algebraic_decay}, Eq.~\eqref{eq:w_more_precise}.          

Next we let $\mathcal{Q}_n$ denote the collection of subsets of $\Lambda_n$ of the form $B_r(o) \cup (\partial B_r(o) \cap A)$, $A \subset \ZZ^d$, $r\le n-1$. These are `cube-like' domains where the boundary may be irregular. The relevance of such domains will become clear in the next paragraph.    
\begin{Thm}\label{thm:sampling_alg}
Suppose that for some $C>0$, for all $\Lambda \in \mathcal{Q}_n$, the $\RFIM$ on $\Lambda$ satisfies $\WSM(C)$. Then for any $\eps > 0$ and $\delta \in (0,1)$ there exists a randomized algorithm taking as input $\beta, h$ and $\Lambda_n$ which runs in time $n^{c}$ for $c=c(\eps,\delta,\beta,C,d)>0$, and outputs a random configuration $\sigma^{\salg} \in \{-1,+1\}^{\Lambda_n}$ with law $\mu_{\Lambda_n}^{\salg}$ such that with $\P_h$-probability $1-\delta$, 
 \begin{equation}\label{eq:tv_sampler}
 d_{\sTV}\big(\mu_{\Lambda_n},\mu_{\Lambda_n}^{\salg}\big) \le \eps \, .
\end{equation}
\end{Thm}

Our sampling algorithm can attain any polynomially small total variation distance to equilibrium in polynomial time. The algorithm proceeds iteratively by running Glauber dynamics 
for polynomially many steps on the $\RFIM$ measure on a domain $V_i$ with free boundary condition, where $V_1=\{o\} \subset \cdots \subset V_N = \In$ is an increasing sequence of domains such that $V_{i+1}$ is obtained from $V_i$ by adding a single vertex $v_{i+1}$ to its boundary. The chain at step $i+1$ on $V_{i+1}$ is initialized from the output of the chain at step $i$ on $V_i$ concatenated with the spin $\sigma_{v_{i+1}}$ sampled independently.    
We prove that at each step $i$, this initialization is a \emph{warm start} with respect to the $\RFIM$ measure on $V_i$, $\mu_{V_{i}}$, i.e., the initialization has bounded Radon-Nikodym derivative w.r.t.\ $\mu_{V_{i}}$. We then use the tools of Lovasz and Simonovits~\cite{lovasz1993random} to show that large set expansion implies the chain mixes in polynomial time. The definition of the collection $\mathcal{Q}_n$ is made exactly to accommodate such a construction: we will have $V_i \in \mathcal{Q}_n$ for all $1\le i\le N$, so that $\WSM$ holds on each $V_i$, and large-set expansion follows from Theorem~\ref{thm:algebraic_decay}. 

  %

\paragraph{Strong spatial mixing.} We now work under \emph{strong spatial mixing} ($\SSM$) in expectation, whereby the influence on the root spin 
decays exponentially in the distance from the vertices \emph{on which the two boundary conditions disagree}, i.e.,  even in the presence of arbitrarily close \emph{agreeing} pinnings. Assuming $\SSM$ in expectation, we can show the stronger result that the continuous-time Glauber dynamics on the $\RFIM$ on $\Lambda_n$ mixes in $N^{o(1)}$ time, or equivalently mixing happens after $N^{1+o(1)}$ spin flips for the discrete chain. 

\begin{Def}\label{assump:SSM-RFIM0}
  For a constant $C>0$, we say that the $\RFIM$ on $\Lambda_n$ satisfies $\SSM(C)$ if for every $r \ge 1$, every box $B\in \{ w+ B_{r}(o): w\in \{-r+1,...,r-1\}^d\}$ (this is all possible cubes of radius $r$ having the origin $o$ in their interior), and every $z\in \partial B$, 
    \begin{align*}
       \EE_h\Big[\max_{\xi \in \{-1,+1\}^{\partial B \setminus \{z\}}} d_{\sTV}\big(\mu_{B}^{\xi^+}(\sigma_o \in\, \cdot \,) \, ,\, \mu_{B}^{\xi^-}(\sigma_o \in\, \cdot \,)\big)\Big] \le Ce^{ - d(o,z)/C}\,,
    \end{align*}
    where $\xi^+, \xi^- \in \{-1,+1\}^{\partial B}$ are the boundary conditions which agree with $\xi$ on $\partial B \setminus \{z\}$ and take the values $+$ and $-$ on $z$ respectively.
\end{Def}

Let $\mE_{\mu_{\Lambda}}$ be the Dirichlet form of Glauber dynamics: for a test function $\varphi: \{-1,+1\}^{\Lambda} \rightarrow \real$,
\begin{equation} \label{eq:dirichlet0}
\mE_{\mu_{\Lambda}} (\varphi, \varphi) := \frac{1}{2} \sum_{\sigma \sim \sigma'}\frac{ \mu_{\Lambda} (\sigma)\mu_{\Lambda} (\sigma')}{\mu_{\Lambda} (\sigma)+\mu_{\Lambda} (\sigma')} \left(\varphi(\sigma) - \varphi(\sigma') \right)^2  \, ,
\end{equation}
 and where the sum is over pairs $\sigma,\sigma' \in \{-1,+1\}^{\Lambda}$ which differ in only one coordinate, and define the spectral gap of the chain as

 \begin{equation} \label{eq:spectralgap}
\gap_{\Lambda} = \inf_{\varphi} \frac{ \mE_{\mu_{\Lambda}} (\varphi, \varphi) }{\var_{\mu_{\Lambda}}(\varphi)} \, ,
\end{equation}
where the infimum is taken over test functions $\varphi : \{-1,+1\}^{\Lambda} \to \real$ of non-zero variance.

Our main result under $\SSM$ is the following. 

\begin{Thm}\label{mainthm:SSM}
    Suppose the $\RFIM$ on $\Lambda_n$ satisfies $\SSM(C)$ for some $C>0$. Then there exists $\kappa = \kappa(d,\beta,C)>0$ such that for every $L \ge \kappa \log n$,
    \begin{equation*}
    \gap_{\In} \ge \exp\big\{-L^{\frac{d-1}{d} + o(1)}\big\} \, ,
    \end{equation*}
    with $\P_{h}$-probability at least $1-e^{-L}$.
In particular, by taking $L = \kappa \log n$, the mixing time of continuous-time Glauber dynamics on $\Lambda_n$ is at most $n^{o(1)}$ with $\P_{h}$-probability at least $1-n^{-\kappa}$. 
\end{Thm}

It is not difficult to show that for every $\beta>0$, $\SSM$ (and hence $\WSM$) holds if the external field magnitudes $|h|$ are sufficiently anti-concentrated: see Lemma~\ref{lem:ssm_large_fields}. 
Conversely, one may wonder what should be the optimal mixing time for the $\RFIM$, and in particular if there is a matching upper bound on the spectral gap to Theorem~\ref{mainthm:SSM}. 
 General arguments in~\cite{Hayes-Sinclair} give a `coupon collecting' lower bound on the mixing time of $\Omega(\log n)$. In Proposition~\ref{prop:mixing-lower-bound}, we show a super-polylogarithmic lower bound on the mixing time for the $\RFIM$ (upper bound on the spectral gap) for $\beta$ sufficiently large, provided the field variances are $O(1)$, confirming the expected slowdown in the Griffiths phase. The proof goes by embedding slow mixing with no external field in a sub-box of side-length $(\log n)^{1/d}$  in $\Lambda_n$ which by chance has very small values of $h$. To be precise, in this Griffiths regime the $L^{(d-1)/d}$ dependence in Theorem~\ref{mainthm:SSM} should be optimal.

\subsection{Overview of the proof}
Our proof is broadly based on a combination of ideas and techniques used to analyze dynamics of lattice models (mainly coarse graining and block dynamics) with the more recently developed techniques for studying mixing times for general spin systems (specifically, stochastic localization). Our sampling algorithm is an adaptation of ideas from algorithmic convex geometry. We describe these components in what follows.

\subsubsection{An easy quasipolynomial bound}\label{sec:easy-quasipolynomial-bound}
 As a warm-up, and to introduce the difficulties inherent to obtaining polynomial convergence in every dimension, let us provide a proof of $e^{O((\log n)^{d-1})}$ mixing on $\In$ under $\WSM$, Def.~\ref{ass:general-WSM}. (This already provides polynomial mixing time for $d=2$.) The crux of this bound is to tile the domain with boxes of side-length $C_0 \log n$ for some large constant $C_0$, at which scale every box exhibits decay of correlation with high probability, and then union bound over the probability of not coupling in all of these boxes after their largest mixing time. 

Let us assume $\WSM(C)$, $C>0$. Using Markov's inequality and a union bound over the vertices in $\Lambda_n$, there exists $C'>0$ such that with probability $1-o_n(1)$, for all  $r\ge C' \log n$, 
\begin{align}\label{eq:quenched-spatial-mixing-assumption}
    \max_{v\in \Lambda_n} ~ d_{\sTV}\Big(\mu_{B_r(v)}^+(\sigma_v \in \cdot) \, , \,  \mu_{B_r(v)}^-(\sigma_v \in \cdot) \Big) \le C' e^{ - r/C'}\,.
\end{align}
Let $r_0 = C_0 \log n$ and let $B_v = B_{r_0}(v)$ for each $v \in \In$. Fix a vertex $v$ and let $(Y_{t,v}^+),\, (Y_{t,v}^-)$ be two chains on the state space $\{-1,+1\}^{B_v}$ following Glauber dynamics transitions, initialized from the all `$+$' and all `$-$' configurations respectively, and ignoring all updates outside $B_v$. Then  
\begin{align}\label{eq:ub_coupling}
    \max_{\sigma_0,\sigma_0'} d_{\sTV}\left(\P(X_t^{\sigma_0}\in \cdot) , \P(X_t^{\sigma_0'}\in \cdot)\right) 
    &\le \sum_{v \in \In} \mathbf{P}\left(Y_{t,v}^+(\sigma_v) \ne Y_{t,v}^-(\sigma_v)\right) \, .
\end{align}
The above inequality is obtained by a union bound over the vertices, and monotonicity of the disagreement probability in the distance from the boundary under the grand coupling $\mathbf{P}$. The latter is the coupling using the same Poisson clock to update both chains, and when a vertex $w \in B_v$ needs updating, $Y_{t,v}^+(w)$ and $Y_{t,v}^-(w)$ are sampled using the same uniform random variable. (See the preliminaries in Section~\ref{sec:preliminaries}.)

By a triangle inequality, each summand on the right-hand side of~\eqref{eq:ub_coupling} is bounded by 
\begin{align}\label{eq:ubd}
   2 \max_{\iota\in \pm} ~d_{\sTV}\left(\P(Y_{t,v}^{\iota} \in \cdot )\,,\, \mu_{B_v}^{\iota}\right) + d_{\sTV}\big(\mu_{B_v}^{+}(\sigma_v \in \cdot) \,,\, \mu_{B_v}^{-}(\sigma_v \in \cdot)\big)\, .
\end{align}
By submultiplicativity of the total variation distance, the first quantity is at most $o(n^{-d})$ for 
\begin{equation*}
t\ge T_0:= (C_d \log n) \max_{v} \big\{ t\mix^+(B_v)\vee t\mix^-(B_v)\big\}\,,
\end{equation*}
where $t\mix^{\iota}(B_v)$ is the mixing time of $(Y_{t,v}^{\iota})_{t \ge 0}$, $\iota = \pm$ (formally defined in Section~\ref{sec:preliminaries}). 
With high probability, the second quantity in~\eqref{eq:ubd} is also $o(n^{-d})$ for all $v$ by~\eqref{eq:quenched-spatial-mixing-assumption} if $C_0$ is a sufficiently large constant. The proof is complete by noting that $T_0 \le e^{O((\log n)^{d-1})}$ by the fact that the cut-width (see for instance Definition 1.4 of~\cite{BKMP}) of $B_v$ is at most $r_0^{d-1} = O((\log n)^{d-1})$ and that the mixing time is no more than exponential in the cut-width. This can be shown using the canonical path method of~\cite{JS}; see Proposition 1.1 in~\cite{BKMP} for further details. 

The weakness of this approach, most apparent for $d \ge 3$, is that the boxes in our tiling have volume $(\log n)^d$ as necessitated by~\eqref{eq:quenched-spatial-mixing-assumption} and the union bound over all vertices of $\In$. 
When the external fields are strong enough, we may expect to do better by adapting the tiling to the realization of the fields. The regions of slow-down  of the chain are regions of weak fields, which by percolation arguments are of volume $\log n$ instead of $(\log n)^d$. 
To implement this, when the external fields are only marginally strong enough for decay of correlations (e.g., the minimal assumption of $\WSM(C)$ in expectation), we artificially increase the external field strength using a \emph{stochastic localization process}. We explain both of these steps in the next three subsections.  

\subsubsection{Weak Poincar\'e under WSM and reduction to strong fields}

We will show that under $\WSM$, the $\RFIM$ on $\In$ satisfies a \emph{weak} Poincar\'e inequality of the form 
\begin{equation}\label{eq:weak_PI1}
     \var_{\mu_{\In}} (\varphi)  \leq K N^{\kappa}\, \mE_{\mu_{\In}}(\varphi,\varphi)^{1/p}\, \osc(\varphi)^{2/q}\, ,
\end{equation}
for all test functions $\varphi$, with high probability over the randomness of the external field $h$, for some $K, \kappa>0$ and $p,q \ge 1$ with $1/p+1/q=1$ which do not depend on $N$; see Theorem~\ref{thm:weak_PI} for the precise statement. We will deduce Theorem~\ref{thm:algebraic_decay} as a consequence of this inequality.
 
 We employ the stochastic localization ($\SL$) technique of~\cite{eldan2013thin} to prove~\eqref{eq:weak_PI1}; we provide a detailed overview of this technique and use for bounding Markov chain mixing times in Section~\ref{sec:preliminaries}.    
We consider the $\SL$ process $(\nu_t)_{t \ge 0}$ with $\nu_0 = \mu_{\In}$. In our context, this process is characterized as follows: letting $\langle \cdot,\cdot \rangle$ be the usual inner product in $\mathbb R^N$, we may define
\begin{equation*}
\nu_t(\sigma) \propto \exp \big\{- H(\sigma) - \langle y_t , \sigma\rangle \big\} \ , 
\end{equation*}
for a stochastic process $y_t$ of the form $y_t = B_t + v_t$ where $B_t$ is a standard Brownian motion in $\R^N$ and $v_t$ is a predictable process (drift). Moreover, the process $v_t$ is the unique process for which $\nu_t(A)$ is a martingale for every subset $A$ of the configuration space.

In other words, the measure $\nu_t$ still has the form of a certain $\RFIM$ on $\Lambda_n$, but the external fields follow a different distribution. We will see that, roughly speaking, the strength of the external fields \emph{increases} with time. For a sufficiently large but constant (in $N$) time $T$ we aim to show two key properties, to be sketched in further detail in the subsections that follow: 
\begin{itemize} 
\item The variance of any test function $\varphi$ under $\nu_T$ is `not much smaller' than its variance under $\nu_0$, with high probability. This statement takes the form (see Theorem~\ref{thm:approx_var_bound}),
\begin{equation}\label{eq:weak_nash0}
     \var_{\nu_0} (\varphi)  \leq KN^{1/q}\, \E\big[\var_{\nu_T}(\varphi) \,|\, h \big]^{1/p}\, \osc(\varphi)^{2/q}\, ,
\end{equation}
for some $p, q$ such that $1/p+1/q=1$, with high $\P_h$-probability. (The expectation in the above display is with respect to the $\SL$ process.)  
A stronger bound where $p=1$, $q=+\infty$ has been referred to as an \emph{approximate variance conservation bound} and has been used to show spectral gaps for Ising measures whose interaction matrix has a spectral diameter bounded by 1~\cite{EKZ,chen2022localization}. In our case we are only able to prove the weak version~\eqref{eq:weak_nash0}; this is due to the fact that the covariance matrix of $\nu_T$ is not almost-surely bounded, but is of typical size $O(\log N)$ with an exponential upper tail; see Theorem~\ref{thm:trace}. 
\item The resulting random measure $\nu_T$ has a spectral gap with high probability, in the usual sense; see Theorem~\ref{thm:fast-mixing-large-variance}. This crucially relies on the fact that the coordinates of the new external field $h+y_T$ have independent magnitudes (though not signs) and can be made sufficiently anti-concentrated by taking $T$ to be a large constant. Thus the regions of weak fields form a subcritical percolation process in $\In$. We construct a (field-dependent) block-dynamics for $\nu_T$ which exploits this property, and this allows us to prove a spectral gap. See Section~\ref{sec:large_variance} for the construction.   
\end{itemize}
Then the weak PI~\eqref{eq:weak_PI1} for $\mu_{\In}$ follows by combining Eq.~\eqref{eq:weak_nash0} with a spectral gap for $\nu_T$, together with the observation that the Dirichlet form process $t \mapsto \mE_{\nu_{t}}(\varphi,\varphi)$
is a super-martingale (see~\cite[Lemma 9]{EKZ}). The statement and the full proof of the weak PI can be found in Section~\ref{sec:weak_PI}, and the proof of Theorem~\ref{thm:algebraic_decay} can be found in Section~\ref{sec:proof_algebraic_decay}. 

\subsubsection{Approximate variance conservation bound}
Towards obtaining \eqref{eq:weak_nash0}, we follow the usual paradigm of the $\SL$ technique which suggests that the rate of variance decay along the $\SL$ process can be bounded by the operator norm of the respective covariance matrices; this is explained below in Section~\ref{sec:preliminaries}. In particular, the decay rate of $\var_{\nu_t}(\varphi)$  
can be controlled via the quantity $\| \cov(\nu_t) \|_{\op}$ where $\|\cdot\|_{\op}$ denotes the operator norm, which is in turn bounded by $\tr(\cov(\nu_t)^p)^{1/p}$ with $p \sim \log n$,  where $\tr$ denotes the trace. 

Controlling the covariance matrix of $\nu_t$ is the heart of the argument, and uses several properties of the $\RFIM$. Defining $(a_{i,j})_{i,j}$ as the entries of $\cov(\nu_t)$, the natural first step is to write
\begin{equation}\label{eq:tracecov}
\tr \big(\cov(\nu_t)^p\big) = \sum_{u_1,...u_p \in \Lambda_n} a_{u_1, u_2} \cdots a_{u_{p-1}, u_p} a_{u_{p}, u_1}\, .
\end{equation}
Now we can control the influence of $u_{i+1}$ on $u_i$ by introducing a boundary at distance $r \le d_{\infty}(u_{i},u_{i+1})$ from $u_i$. A simple argument based on the $\FKG$ inequality, see Lemma~\ref{lem:cov_inf}, yields 
\begin{equation}\label{eq:TV-influence-bd}
a_{u_i,u_{i+1}} \leq  d_{\sTV} \big( \nu_t(\sigma_{u_i} \in \,\cdot\, |\, \sigma_{\partial B_r(u_i)}=+)\, ,\, \nu_t(\sigma_{u_i} \in \,\cdot\,|\, \sigma_{\partial B_r(u_i)}=-) \big)\, .
\end{equation}
The $\WSM$ condition, Definition~\ref{ass:general-WSM}, then gives an upper bound on the right-hand side in the case $t=0$. 
 Obtaining a similar bound for general $t$ is more involved. It relies on two main components:
 \begin{itemize}
 \item
 A calculation that shows that the total-variation distance in~\eqref{eq:TV-influence-bd} is in some sense a super-martingale with respect to $t$. This calculation is somewhat similar to that showing the Dirichlet form is a super-martingale~\cite{EKZ}, and crucially uses the $\FKG$ correlation inequality. 
 \item 
 A representation theorem obtained in \cite{AhmedMontanari} shows that the measure $\nu_t$ can be written in a form that implies that the total variations at different points are essentially independent under $\WSM$. This allows to control the expectation of the product in~\eqref{eq:tracecov} under the randomness of the $\SL$ process. 
\end{itemize} 
 
 Finally, a combinatorial counting argument is used in conjunction with the above estimates to obtain a bound for $\tr (\cov(\nu_t)^p)$.

\subsubsection{Coarse-graining and disorder-dependent block dynamics}
After use of the stochastic localization, we need to show a high probability bound on the spectral gap of the Glauber dynamics for the $\RFIM$ when the external fields are sufficiently anti-concentrated. We show that under such an assumption, the inverse spectral gap is at most $N^{o(1)}$. This is actually proved in conjunction with Theorem~\ref{mainthm:SSM} as they both use identical strategies based on combining a field-dependent \emph{coarse-graining} with a block dynamics where the blocks are non-homogenous and determined by the realization of the field.

Our aim is to localize the dynamics to blocks of diameter $C \log n$, inside of which the large typical size of the field shields the center of the ball from the influence of the boundary. This would reduce the mixing time on $\Lambda_n$ to the maximum mixing time on a block. However, as indicated by the short proof in Subsection~\ref{sec:easy-quasipolynomial-bound}, if these regions are chosen as cubes, the worst-case mixing time will be super-polynomial in $n$ in $d\ge 3$.  

Our proof leverages the fact that the regions where the external field is not large are a sparse subset of $\Lambda_n$, i.e., they behave like a subcritical percolation and their largest connected components have $O(\log n)$ volume. We construct  localizing blocks depending on the realization of $h$, in such a way that their largest volume is $\tilde O(\log n)$, their boundary is sufficiently shielded from their bulk by large fields, and the blocks' worst mixing time is sub-polynomial. The crux of this construction is twofold: 
\begin{itemize}
    \item We coarse-grain the lattice into good and bad blocks of radius $\log \log n$, where bad blocks are those dominated by small field values (in a quantifiable way). The coarse-graining enables a construction which isolates the ``bad'' portions of $\Lambda_n$ while building up enough buffer around them that different boundary conditions on a block do not influence its bulk. The largest blocks in this construction will have volume $O(\log n \log \log n)$. 
    \item The amenability of and isoperimetry in $\mathbb Z^d$ ensure that the mixing time of Glauber dynamics on an Ising model on a domain $B$ with volume $|B|$ is at most $\exp(O(|B|^{(d-1)/d}))$; see Lemma~\ref{lem:surface-mixing-time-bound}. This then gives a sub-polynomial worst-case mixing time over the blocks.
\end{itemize}
We note that similar coarse-graining ideas were used to homogenize quenched randomness in~\cite{CMM-Disordered-magnets} for the random ferromagnet (random couplings rather than fields).

\subsection*{Acknowledgements}
The authors thank anonymous referees for their careful reading and useful comments.
The research of R.G. was supported in part by NSF DMS-2246780. A.P. was supported by the Israeli Council for Higher Education (CHE) via the Weizmann Data Science Research Center, and by a research grant from the Estate of Harry Schutzman Science Research Fellowship.

\section{Preliminaries}
\label{sec:preliminaries}
\paragraph{Markov chains} We first recall some basics of Markov chain mixing times used in our context. The continuous time chain $(X_t)$ defined in~\eqref{eq:glauber} is reversible, irreducible and aperiodic with stationary measure $\mu_{G}$.  We define its $\varepsilon$-mixing time
    \begin{equation*}
        \tau\mix^\cont(\varepsilon):=\inf \left\{ t\geq 0~: \max_{\sigma\in \{-1,+1\}^{V}} d_{\sTV}\big(\P(X_t \in \cdot \,|\, X_0 =\sigma)\,,\,\mu_G\big)\leq \varepsilon \right\} \, ,
    \end{equation*}
where $d_{\sTV}(\cdot,\cdot)$ is the total variation distance between probability measures. When we omit to refer to the parameter $\eps$ in the above definition, it is implicitly understood that $\eps$ is a constant, e.g., $\eps = 1/4$. 
 We sometimes abuse notation and write $d_{\sTV}(X,Y)$ or $d_{\sTV}(X,\nu)$ for two random variables $X\sim \mu$ and $Y\sim \nu$ instead of $d_{\sTV}(\mu,\nu)$.  
The $\eps$-mixing time can be estimated in terms of the inverse spectral gap as defined in~\eqref{eq:spectralgap}, see \cite[Theorem 20.6]{LP} :
    \begin{equation*}
        \tau\mix^\cont(\varepsilon)\leq \frac{1}{\gap_{G}}\log\left(\frac{1}{\varepsilon \mu_{\min}} \right)\,,
    \end{equation*}
    where  $\mu_{\min}=\min_{\sigma\in\{-1,+1\}^{V}} \mu_{G}(\sigma)$.

The discrete Glauber dynamics chain chooses a vertex $v \in V$ uniformly at random at every step and resamples its value according to the same conditional distribution~\eqref{eq:glauber}. Letting $P$ denote its transition matrix, its $\varepsilon$-mixing time $\tau\mix(\varepsilon)$ is similarly defined as the smallest integer $k \ge 0$ such that $d_{\sTV}( P ^k(\sigma,\cdot),\mu) \leq \varepsilon$ for all starting points $\sigma \in \{-1,+1\}^V$.    
The two mixing times are related by a factor of $|V|$ up to constants (see, e.g.\ \cite[Theorem 20.3]{LP}):
      \begin{equation*}
       (\varepsilon/2)\,|V| \,\tau\mix^\cont(2\varepsilon)  \le \tau\mix(\varepsilon) \le 2 |V| \,\tau\mix^\cont(\varepsilon/4) \, .
        \end{equation*}

\paragraph{FKG and couplings} Positivity of the $\RFIM$ interactions ($\beta \ge 0$) implies useful correlation inequalities such as the $\FKG$ inequality (see e.g.,~\cite{friedli2017statistical,duminil2017lectures}). In particular, for any subset $A \subset V$ and any non-decreasing function $\varphi : \{-1,+1\}^A \to \real$, the mean of $\varphi$ under $\mu^{\tau}_{A}$ is monotone in the boundary condition: $\mu^{\tau}_{A}(\varphi) \le \mu^{\tau'}_{A}(\varphi)$ if $\tau \le \tau'$ (where this inequality is understood vertexwise). This for instance implies that the covariance under $\mu_{G}$ between any two non-decreasing functions, on the natural partial order over Ising configurations, is non-negative.

Given two boundary conditions $\tau, \tau'$, the measures  $\mu^{\tau}_{A},\, \mu^{\tau'}_{A}$ can be coupled via the \emph{monotone coupling} where we fix an ordering $v_1,\cdots,v_k$ of the vertices in $A$ and draw unit uniform random variables $U_1,\cdots,U_k$ independently. We then construct the coupling $(\sigma^{\tau}, \sigma^{\tau'})$ sequentially where for each $i \ge 1$ we sample $\sigma_{v_i}^{\tau}$ and $\sigma_{v_i}^{\tau'}$ from their respective conditional distributions (given $(\sigma_{v_j}^{\tau})_{ j < i}, \tau$  and $(\sigma_{v_j}^{\tau'})_{ j < i}, \tau'$ respectively) using $U_i$ as  common randomness. It can be verified that this coupling is monotone in the boundary condition: $\sigma^{\tau} \le \sigma^{\tau'}$  if   $\tau \le \tau'$. 

Finally, two chains of (continuous time) Glauber dynamics can be coupled via \emph{the grand coupling} where the same Poisson clock is used to update both chains, and whenever a vertex is to be updated (simultaneously in both chains), the same uniform random variable is used. This coupling is monotone in both initialization and boundary condition.

\paragraph{Stochastic localization}
 We now present some fundamental properties of \emph{stochastic localization} ($\SL$), first introduced in \cite{eldan2013thin}. In words, $\SL$ is  a diffusion $(\nu_t)_{t\ge 0}$ in the space of probability measures on $\{-1,1\}^N$ such that its $t\to\infty$ limit is a delta mass at a random $\sigma^*\sim \nu_0$: i.e., it is randomly picking a $\sigma_* \sim \nu_0$ to localize about. Since its introduction, SL has found many applications for solving central problems in high-dimensional geometry (e.g.,~\cite{Klartag-Lehec,Chen}) and more recently, mixing times of Markov chains and sampling algorithms for Ising models and spin glasses (e.g.,~\cite{EKZ,chen2022localization,EMS}). In what follows, we will give a formal definition, describe some properties and explain how it is typically used for mixing time and spectral gap bounds.

Let $\nu_0$ be a probability distribution on the hypercube $\mathcal{C}_N:=\{-1,+1\}^N$. (In our application, $\nu_0 = \mu_{\Lambda}$ will be the $\RFIM$ measure on a domain $\Lambda$ of interest.)
 Let $(B_t)_{t\geq 0}$ be a standard Brownian motion on $\mathbb{R}^N$ with $B_0=0$. $\SL$ is the stochastic process $(\nu_t(\sigma))_{t \geq 0}$, whose density with respect to $\nu_0$
  \begin{equation*}
  \frac{\dd\nu_t}{\dd \nu_0}(\sigma)=F_t(\sigma) 
 \end{equation*}
 solves the stochastic differential equation 
\begin{equation}\label{eq:SL_SDE}
\begin{cases}
    \dd F_t(\sigma)&=F_t(\sigma)\langle  \sigma-a_t,\dd B_t\rangle \quad\quad  \forall \sigma \in \mathcal{C}_N\,,\\
    F_0(\sigma)&=1 \, ,
\end{cases}
\end{equation}
where $a_t$ is the barycenter of, or mean of a sample from, $\nu_t$: 
\begin{equation}\label{eq:barycenter_def}
    a_t=\int_{\mathcal{C}_N}\sigma\,\, \nu_t(\dd \sigma)\,.
\end{equation}
By \cite[Proposition 9]{RonenHypercube}, the $\SL$ process defined in \eqref{eq:SL_SDE} ensures that almost surely, for all \( t \), \(\nu_t\) is a probability measure as 
\begin{equation*}
    \dd \nu_t (\mathcal{C}_N) = \int_{\mathcal{C}_N} F_t(\sigma) \langle \sigma-a_t, \dd B_t \rangle \dd \nu_t(\sigma) = 0\, .
\end{equation*}
Moreover, for any \( A \subseteq \mathcal{C}_N \), the function \( t \mapsto \nu_t(A) \) is a martingale, as evident from the integral form of \eqref{eq:SL_SDE}. By the same remark, the process $a_t$ is also a martingale and almost surely converges to a point in $a_\infty:= \lim_{t\rightarrow\infty} a_t \in \mathcal{C}_n$, which is distributed according to the law $\nu_0$. The measure $\nu_t$ almost surely weakly converges to a Dirac measure at $a_\infty$.

The use of the $\SL$ process in proving Poincar\'e-type inequalities for $\nu_0$ is described as follows. The aim is to reduce the Poincar\'e inequality for $\nu_0$ to one on $\nu_T$ for a suitable stopping time $T$, which is ideally easier to prove. Let $\varphi:\mathcal{C}_N\rightarrow \mathbb{R}$ be a test function. On the one hand, the Dirichlet form is a supermartingale under the SL process: recalling the definition~\eqref{eq:dirichlet0} of the Dirichlet form, one can verify via It\^o's formula that the process $t \mapsto \nu_t(\sigma)\nu_t(\sigma')/(\nu_t(\sigma)+\nu_t(\sigma'))$ has negative drift, as done in~\cite[Lemma 9]{EKZ}. A heuristic understanding of this is that ``variance-like" quantities tend to shrink under the $\SL$ process.   

On the other hand, if we let $M_t$ be the expectation of $\varphi$ under $\nu_t$, i.e., $M_t = \int_{\mathcal{C}_N} \varphi \dd\nu_t$, then the variance of $\varphi$ satisfies 
\begin{equation*}
    \var_{\nu_0}(\varphi)=\mathbb{E}[M]_t+\mathbb{E}\var_{\nu_t}(\varphi)\,.
\end{equation*}
where $[M]_t$ denotes the quadratic variation. 
If one shows that $\var_{\nu_T}(\varphi)$ is not much smaller than $\var_{\nu_0}(\varphi)$ then a Poincar\'e inequality on $\nu_T$ can be translated to one on $\nu_0$. 
The heart of the stochastic localization technique is to relate the contraction of the variance of $\varphi$ to a certain covariance matrix, namely, the contraction turns out to be governed by the inequality
\begin{equation*}
    \dd [M]_t \leq \var_{\nu_t}(\varphi) \cdot  \| \cov(\nu_t)\|_\op\, .
\end{equation*}
The above follows from an application of It\^o's formula and the Cauchy-Schwartz inequality; the derivation is in Eq.~\eqref{eq:[M]_t} below. Using the above, one gets 
\begin{equation*}
    \dd \var_{\nu_t}(\varphi) \geq -\|\cov(\nu_t)\|_\op \var_{\nu_t}(\varphi)\dd t + \text{martingale} \, .
\end{equation*}
An almost-sure upper bound $\|\cov(\nu_t)\|_\op \le \alpha(t)$ on the covariance would then yield, after integration, the approximate variance conservation bound
\begin{equation*}
    \frac{\EE{\var_{\nu_T}(\varphi)}}{\var_{\nu_0}(\varphi)}\geq  \exp\left(-\int_0^T \alpha(t) \dd t \right)\, .
\end{equation*}
This then yields a Poincaré inequality for $\nu_0$ in terms of the Poincar\'e inequality (with constant $C_P$) for $\nu_T$:
\begin{align*}
    \mathcal{E}_{\nu_0}(\varphi,\varphi) \geq \EE \mathcal{E}_{\nu_T} (\varphi,\varphi) 
    &\geq C_P \EE \var _{\nu_T}(\varphi) \\
    &\geq C_P \var_{\nu_0}(\varphi) \exp\left(-\int_0^T \alpha(t) \dd t \right) \,,  
\end{align*}
where in the first inequality we use the fact that the Dirichlet form is a supermartingale. This is the bound given in, e.g.,~\cite[Theorem 49]{chen2022localization}.

When looking to bound $\cov(\nu_t)$ along the $\SL$ process, and when proving the Poincar\'e inequality on the large $T$ measure $\nu_T$, it ends up being very helpful that the measures $\nu_t$, $0 \le t \le T$, are themselves Ising models with random external fields. 
Namely, the $\SL$ process can be seen as a linear stochastic tilt of $\nu_0$ as follows. 
For an external field $y \in \real^N$, consider the tilted measure
\begin{equation*}
   \mu_{y}(\sigma) \propto e^{\langle y, \sigma\rangle} \nu_{0}(\sigma)  \, .
   \end{equation*}  
 The stochastically localized measure can be written as 
 \begin{equation}\label{eq:tilt} 
     \nu_t(\sigma) =\mu_{y_t} (\sigma)\propto e^{\langle y_t, \sigma\rangle} \nu_0(\sigma)  \, ,
 \end{equation} 
 where we let $y_t$ evolve according to the SDE
\begin{equation}\label{eq:sde_sl} 
  \rmd y_t = a(y_t) \rmd t + \rmd B_t \, , ~~~~ y_0 = 0\, ,
  \end{equation}
with $a(y_t) = a_t$ as in \eqref{eq:barycenter_def}.

One can get further insight into the nature of tilted measures $\nu_t$ via the following Bayesian interpretation of the SL process:
\begin{Prop}(\cite[Theorem 2]{AhmedMontanari} and \cite[Proposition 4.1]{BoazEli})\label{prop:sl_bayes}
    Let $(\bar{B}_t)_{t \geq 0} $ be a standard Brownian motion and sample $\sigma ^\ast \sim \nu_0$ independently from $(\bar{B}_t)_{t \geq 0} $. Then $(y_t)_{t\geq 0}$ defined as in \eqref{eq:sde_sl} has the same distribution of $(\bar{y}_t)_{t\geq 0}$ defined as
\begin{equation} \label{eq:sl}
\bar{y}_t = t \sigma^* + \bar{B}_t \, ,~~~~ t \ge 0 \, ,
\end{equation}
and $(\nu_t)_{t\ge 0}$ is distributed as the process of conditional measures $\big(\nu_0(\sigma^* \in \cdot \, | \, \bar{y}_t)\big)_{t \ge 0}$. 
\end{Prop}

From the above proposition and Eq.~\eqref{eq:tilt} it is clear that when $\nu_0 = \mu_{G}$, $\nu_t$ is another $\RFIM$ measure on $G$ with external field $h + \bar{y}_t$; in particular, as $t$ gets large it is an $\RFIM$ with a large external field, for which it is easier to prove a Poincar\'e inequality. Importantly, however, even with $\WSM$ in expectation on $\nu_0$, there is no useful almost sure bound on $\|\cov(\nu_t)\|_\op$, as $\nu_t$ is moving through the space of $\RFIM$ measures and there is some chance that the field becomes atypically small, leading to an explosion of this term. Much of the work in our Section~\ref{sec:covariance} is towards obtaining good moment control under $\WSM$ on this covariance along the $\SL$ process.

\section{A weak Poincar\'e inequality under WSM}
\label{sec:weak_PI}
Our main aim is to prove a weak Poincar\'e inequality for the $\RFIM$ on a domain $\Lambda$ under $\WSM$, to obtain Theorem~\ref{thm:algebraic_decay} as a consequence. In this section, we provide this proof modulo two important inputs, one being a bound on the covariance matrix of $\nu_t$ under $\WSM$, and the other being a proof of a full Poincar\'e inequality if the fields are sufficiently anti-concentrated (which will be the output of the localization scheme after a large but constant time.) Those two steps will be the content of Sections~\ref{sec:covariance}--\ref{sec:large_variance} respectively.

\subsection{Approximate variance conservation}
We start with our weak version of an approximate variance conservation bound:

\begin{Thm}\label{thm:approx_var_bound}
  Suppose the $\RFIM$ on $\Lambda$ satisfies $\WSM(C)$ for some $C>0$, then there exists $c_0 = c_0(d,C)>0$ such that for all $\delta>0$, 
the following holds with $\P_h$-probability at least $1-\delta$.
For all $T>0$ and all test functions $\varphi:\{-1,+1\}^{\Lambda} \to \R$, we have
     \begin{equation}\label{eq:approx_var_bound}
         \var_{\nu_0}(\varphi) \leq \big(e^{-c_0}N/\delta\big)^{1/q} \, \EE\big[\var_{\nu_T}(\varphi)\, |\, h\big]^{1/p}\,\osc(\varphi)^{2/q}\, ,
     \end{equation}
     with $N = |\Lambda|$, $p = e^{T/c_0}$ and $1/p+1/q=1$.
\end{Thm}

The proof of the above theorem crucially relies on a tail bound for the operator norm of the covariance matrix of $\nu_t$ which we state here:  
\begin{Thm} \label{thm:trace}
    Suppose the $\RFIM$ on $\Lambda$ satisfies $\WSM(C)$ for some $C>0$, then there exists $C_0 = C_0(d,C)>0$ such that for any $p \geq 1$,  
    \begin{equation}\label{eq:trp}
        \EE_h\left[\sup_{t \ge 0} \, \EE \Big[\tr(\cov(\nu_t)^p) \,\big|\, h \Big] \right] \leq (C_0 p)^p\, N\,,
    \end{equation}
    with  $N = |\Lambda|$.
    In particular for all $R\ge 0$,
    \begin{equation}\label{eq:devation_bd}
         \EE_h\left[\sup_{t \ge 0} \, \P\Big( \|\cov(\nu_t)\|_{\op} \ge R \,\big|\, h \Big) \right]   \leq  N e^{-c_0 R}\, , 
    \end{equation}
    with $c_0 = (eC_0)^{-1}$.
\end{Thm}

 The proof of the above theorem is combinatorial in nature and is deferred to Section~\ref{sec:covariance}. 
  
\begin{proof}[Proof of Theorem~\ref{thm:approx_var_bound}]
 We proceed by establishing a variational inequality for $\E \big[\var_{\nu_t}(\varphi) | h\big]$. 
 First we note that it follows from Theorem~\ref{thm:trace} and Markov's inequality that
\begin{equation}\label{eq:cov_bound2}
          \sup_{t \ge 0}\, \P\Big( \|\cov(\nu_t)\|_{\op} \ge R \, | \, h\Big) \leq  N \delta^{-1} e^{-c_0 R}\, ,
\end{equation}
with $\P_{h}$-probability $1-\delta$.
 We consider $h$ and  $\delta>0$ fixed throughout and condition on the event given in Eq.~\eqref{eq:cov_bound2}.  For ease of notation it is implicitly understood that all expectations are conditional on $h$, i.e., we only take expectation with respect to the $\SL$ process. 

Define the martingale $M_t := \nu_t(\varphi)$. Using the definition of the $\SL$ process Eq.~\eqref{eq:SL_SDE}, 
\begin{align*}
    \dd M_t &= \dd \int_{\{-1,+1\}^{\Lambda}} \varphi(x) \nu_t(\dd x) = \int_{\{-1,+1\}^{\Lambda}} \varphi(x) \dd F_t(x) \nu_0(\dd x) \\
    &= \int_{\{-1,+1\}^{\Lambda}} \varphi(x) \langle x-a_t, \dd B_t \rangle \nu_0(\dd x)\, , 
\end{align*}
 and therefore we have 
\begin{align}
    \dd [M]_t &= \left\| \int_{\{-1,+1\}^{\Lambda}} \varphi (x)  (x-a_t)  \nu_t(\dd x) \right\|^2 \dd t \nonumber\\
    &= \sup_{\lvert \theta \lvert = 1} \left( \int_{\{-1,+1\}^{\Lambda}} \varphi(x) \langle x-a_t, \theta \rangle  \nu_t(\dd x) \right)^2 \dd t \nonumber\\
    & \overset{(\ast)}{\leq} \var _{\nu_t}(\varphi) \sup_{\| \theta \| = 1} \int_{\{-1,+1\}^{\Lambda}} \langle x-a_t, \theta \rangle^2  \nu_t(\dd x) \, \dd t \nonumber \\
    &=  \var _{\nu_t}(\varphi) \, \big\|\cov(\nu_t)\big\|_{\op} \, \dd t\, ,\label{eq:[M]_t}
\end{align}
where in $(\ast)$ we apply the Cauchy–Schwarz inequality, and $\cov(\nu_t)$ refers to the covariance matrix of $\nu_t$:
\begin{equation*}
\cov(\nu_t) = \int_{\{-1,+1\}^\Lambda} \sigma \sigma^{\top} \nu_t(\dd \sigma) - a_t a_t^\top \, ,~~~ a_t = \int_{\{-1,+1\}^\Lambda} \sigma \, \nu_t(\dd \sigma)\, .
\end{equation*}
Let us write $A_t :=\cov (\nu_t)$. 
  Next, we  note that $\nu_t(\varphi^2)$ is a martingale. It follows from It\^{o}'s formula that
\begin{align} 
    \dd \var_{\nu_t} (\varphi) &= \dd \left( \nu_t(\varphi^2) - M_t^2 \right) =\dd \nu_t(\varphi^2) -2M_t\dd M_t-\frac{1}{2}2\dd[M]_t \nonumber \\
    &= -\dd [M]_t + \text{martingale}\label{eq:dvar2} \, .
\end{align}
It follows from Eq.~\eqref{eq:dvar2} and Eq.~\eqref{eq:[M]_t} that
\begin{align}
\frac{\dd}{\dd t} \E\big[\var_{\nu_t} (\varphi)\big] &\ge  - \E\Big[\var _{\nu_t}(\varphi)    \|A_t\|_{\op} \Big] \nonumber\\
&\ge -\E\Big[\var _{\nu_t}(\varphi)  \|A_t\|_{\op}  \one_{\|A_t\|_{\op} < R} \Big] -
\E\Big[\var _{\nu_t}(\varphi)  \|A_t\|_{\op}  \one_{\|A_t\|_{\op} \ge R} \Big] \, ,
\label{eq:dvar3}
\end{align}
for every $R>0$. We bound the two terms as 
\begin{align*}
\E\Big[\var _{\nu_t}(\varphi)  \|A_t\|_{\op}  \one_{\|A_t\|_{\op} < R} \Big] \le R \, \E\big[\var _{\nu_t}(\varphi) \big]\, ,
\end{align*}
and 
\begin{align*}
\E\Big[\var _{\nu_t}(\varphi)  \|A_t\|_{\op}  \one_{\|A_t\|_{\op} \ge R} \Big] &\le \osc(\varphi)^2 \E\big[  \|A_t\|_{\op}  \one_{\|A_t\|_{\op} \ge R} \big]\\
&= \osc(\varphi)^2 \Big(R \, \P\big( \|A_t\|_{\op} \ge R\big) + \int_R^{\infty} \P\big( \|A_t\|_{\op} \ge r\big) \rmd r \Big)\\
&\le  \osc(\varphi)^2 N \delta^{-1} (R + c_0^{-1}) e^{-c_0 R} \, ,
\end{align*}
where the last line follows from Eq.~\eqref{eq:cov_bound2}.  We choose 
\begin{equation*}
R = c_0^{-1}\log \left(\frac{N \osc(\varphi)^2}{\delta \E\big[\var _{\nu_t}(\varphi) \big]}\right) \, .
\end{equation*}
 From Eq.~\eqref{eq:dvar3} we obtain the differential inequality 
\begin{equation}\label{eq:ode1}
\frac{\dd}{\dd t} \E\big[\var_{\nu_t} (\varphi)\big] \ge - c_0^{-1} \left(2\log \left(\frac{N \osc(\varphi)^2}{\delta \E\big[\var _{\nu_t}(\varphi) \big]}\right)+1\right)  \E\big[\var_{\nu_t} (\varphi)\big] \, .
\end{equation}
With 
\begin{equation*}
u(t) := \E\big[\var_{\nu_t} (\varphi)\big] \big/ \osc(\varphi)^2 \, ,
\end{equation*}
 Eq.~\eqref{eq:ode1} becomes
\begin{equation*}
\frac{\dd}{\dd t} \log u(t) \ge 2c_0^{-1}\log u(t) - c_0^{-1} \big(1+ 2\log (N/\delta)\big) \, .
\end{equation*}
Considering the function $v(t) = e^{-2t/c_0} \log u(t)$ and integrating, we obtain 
\begin{align*}\label{eq:ode4}
 u(t) &\ge  u(0)^{e^{2t/c_0}} (N/\delta)^{1-e^{2t/c_0}} \exp\Big( - c_0 \big(e^{2t/c_0} -1\big)/2 \Big)\, .
\end{align*}
This implies the claimed bound.
\end{proof}

\subsection{A weak Poincar\'e inequality}
In this section we prove the weak Poincar\'e inequality for the $\RFIM$ on $\Lambda$ under $\WSM$. 
\begin{Thm}\label{thm:weak_PI}
Suppose $\WSM(C)$ holds for the $\RFIM$ on $\Lambda$ for some constant $C>0$. Fix $\delta \in (0,1)$. With $\P_h$-probability at least $1-\delta$, the following holds for all $\varphi:\{-1,+1\}^{\Lambda} \to \R$:
\begin{equation} \label{eq:wPI_nu}
         \var_{\mu_{\Lambda}}(\varphi) \leq  K \delta^{-1/q} N^{\kappa} \,\mE_{\mu_{\Lambda}}(\varphi,\varphi)^{1/p} \cdot \osc(\varphi)^{2/q} \, ,
\end{equation}
where $p,q \ge 1$, $1/p+1/q=1$, and $K, \kappa >0$, all depending on $C,\beta,d$.  
\end{Thm}

Before starting the proof we quote the key result from Section~\ref{sec:large_variance} that a $\RFIM$ with a strong external field obeys a spectral gap inequality with high probability:

\begin{Thm}\label{thm:fast-mixing-large-variance0}
    For every $d,\beta>0$ and $0<\varepsilon_d<1/d$, there exist $\epsilon,K,\kappa$ such that if we consider the $\RFIM$ on $\Lambda$ with external field $h$ such that $(|h_x|)_{x \in \Lambda}$ are i.i.d., with $\P(|h_x|\le K) \le \epsilon$, then for all $L\ge \kappa \log N$,
    \begin{align*}
        \gap_{\Lambda}^{-1} \le e^{L^{1-\varepsilon_d}} \,.
    \end{align*} 
    with probability at least $1- e^{-L}$.
\end{Thm}

\begin{proof}[Proof of Theorem~\ref{thm:weak_PI}]
First as per Proposition~\ref{prop:sl_bayes}, the external field of the $\RFIM$ measure $\nu_T$ is  $y_T + h \stackrel{d}{=} T\sigma^* + B_T + h$, where $\sigma^* \sim \nu_0 = \mu_{\Lambda}$, and $(B_t)$ is an independent standard Brownian motion. Since $\sigma^*$ is a binary vector and the distribution of $h$ is symmetric, the vector of magnitudes $|y_T+h|$ (with absolute value applied entrywise) is distributed as $|T+B_T+h|$ and hence has independent coordinates. Next, let $\epsilon, K, \kappa$ as in Theorem~\ref{thm:fast-mixing-large-variance0}. Letting $Z \sim N(0,1)$, we have for any $H>0$,
\[\P\big(|T+\sqrt{T}Z+h_x| \le K\big) \le \P\big(|T+\sqrt{T}Z| \le K + H\big)+ \P\big(|h_x| \ge H\big)\, .\]
We choose $H$ and $T$ large enough so that the above is smaller than $\epsilon$. (The claim that the second term above can be made small follows by tightness of any single probability measure on $\R$; see~\cite[Theorem 1.3]{billingsley2013convergence}.)

By Theorem~\ref{thm:fast-mixing-large-variance0}, for all $L \ge \kappa\log N$ the following Poincar\'e inequality holds with probability at least $1-e^{-L}$:   
\begin{equation} \label{eq:gap_nuT2}
\var_{\nu_T}(\varphi) \le e^{L^{1-\eps_d}} \mE_{\nu_T} (\varphi, \varphi) \, ,~~~~\mbox{for all } \varphi : \{-1,+1\}^{\Lambda} \to \R\, .
\end{equation}

Now let $E_L$ denote the event ``Eq.~\eqref{eq:gap_nuT2} holds" for some $L \ge \kappa \log N$ to be chosen later. By Markov's inequality, $\P(E_L^c \,|\, h) \le e^{-L}/\delta$ with $\P_h$-probability at least $1-\delta$. 
Then with the same probability it holds that for any $\varphi$ possibly depending on $h$ but not on $y_T$,  
\begin{align} 
\E\big[\var_{\nu_T}(\varphi)\, |\, h\big] &\le e^{L^{1-\eps_d}} \,\E\big[\mE_{\nu_T}(\varphi,\varphi) \one_{E_L} \, |\, h\big] + \E\big[\var_{\nu_T}(\varphi) \one_{E_L^c}\, |\, h\big] \nonumber\\
&\le e^{L^{1-\eps_d}} \,\mE_{\nu_0}(\varphi,\varphi) + \delta^{-1} e^{-L} \osc(\varphi)^2\, .\label{eq:g}
\end{align}
The last inequality follows from $\var_{\nu_T}(\varphi) \le \osc(\varphi)^2$ and the fact that $t \mapsto \mE_{\nu_t}(\varphi,\varphi)$ is a super-martingale (see~\cite[Lemma 9]{EKZ}).
Next, let us assume that
\begin{equation}\label{eq:small_E}
\mE_{\nu_0}(\varphi,\varphi) \le N^{-\kappa} \osc(\varphi)^2\, , 
\end{equation}
since otherwise, 
\begin{equation}\label{eq:sPI}
\var_{\nu_0}(\varphi) \le \osc(\varphi)^2 \le N^{\kappa} \, \mE_{\nu_0}(\varphi,\varphi)\, ,
\end{equation}
and we are done. We then choose
\begin{equation*}
L = \log \left(\frac{\osc(\varphi)^2}{\delta\mE_{\nu_0}(\varphi,\varphi)}\right) \ge \kappa\log N \, .
\end{equation*}
Since for any $\eps>0$,
\begin{equation*}
e^{L^{1-\eps_d}} \le \left(\frac{\osc(\varphi)^2}{\delta \mE_{\nu_0}(\varphi,\varphi)}\right)^{\eps} \, ,
\end{equation*}
provided $N \ge N_0$ for some $N_0 = N_0(d,\eps)$, Eq.~\eqref{eq:g} implies
\begin{align*} 
\E\big[\var_{\nu_T}(\varphi)\, |\, h\big] 
&\le  \delta^{-\eps} \mE_{\nu_0}(\varphi,\varphi)^{1-\eps} \osc(\varphi)^{2\eps} + \mE_{\nu_0}(\varphi,\varphi) \\
& \le \big(\delta^{-\eps} + N^{-\kappa \eps}\big) \, \mE_{\nu_0}(\varphi,\varphi)^{1-\eps} \osc(\varphi)^{2\eps} \, ,
\end{align*}
where the last line follows from~\eqref{eq:small_E}. 
Combining the above with the weak approximate variance conservation bound of Theorem~\ref{thm:approx_var_bound} we obtain with $\P_h$-probability at least $1-2\delta$, for all $\varphi:\{-1,+1\}^{\Lambda} \to \R$, 
\begin{equation} \label{eq:wPI_nu01}
         \var_{\nu_0}(\varphi) \leq  K \big(\delta^{-\eps} + N^{-\kappa \eps}\big)^{1/p}(N/\delta)^{1/q} \,\mE_{\nu_0}(\varphi,\varphi)^{(1-\eps)/p} \cdot \osc(\varphi)^{2\eps/p+2/q} \, ,
\end{equation}
where $p = e^{2T/c_0}$, $1/p+1/q=1$, and $K = e^{-c_0/(2q)}$. We ignore the term $N^{-\kappa \eps}$ and take $p' = (1-\eps)/p$, $1/q' = \eps/p + 1/q$. We obtain the desired conclusion by combining Eq.~\eqref{eq:sPI} with Eq.~\eqref{eq:wPI_nu01} into one statement and adjusting the constant $\kappa$.     
\end{proof}

\subsection{Polynomial relaxation and construction of the sampling algorithm}
\label{sec:proof_algebraic_decay}
In this section we prove Theorems~\ref{thm:algebraic_decay} and Theorem~\ref{thm:sampling_alg}.

\subsubsection{Polynomial relaxation and mixing from warm start}
\begin{proof}[Proof of Theorem~\ref{thm:algebraic_decay}] The polynomial relaxation is a direct consequence of the weak PI as shown in the paper of Liggett~\cite{liggett1991l_2}. We reproduce the argument here for completeness.    
From Theorem~\ref{thm:weak_PI} (applied to $\Lambda = \Lambda_n$) we have with $\P_h$-probability $1-\delta$,
\begin{equation} \label{eq:wPI_nu2}
         \var_{\mu_{\Lambda}}(\varphi) \leq  A\,\mE_{\mu_{\Lambda}}(\varphi,\varphi)^{1/p} \cdot \osc(\varphi)^{2/q} \, ,
\end{equation}
for all $\varphi$, with $A = K \delta^{-1/q} N^{\kappa}$.
Recalling that $P_t$ is the Markov semigroup of continuous-time Glauber dynamics, we have 
\begin{align*} \label{eq:t_decay0}
        \frac{\rmd}{\rmd t} \var_{\mu_{\Lambda}}(P_t\varphi) &= -2 \mE_{\mu_{\Lambda}}(P_t\varphi,P_t\varphi)  \\
        &\le   -2  \frac{\var_{\mu_{\Lambda}}(P_t\varphi)^p}{A^p\osc(P_t\varphi)^{2p/q}} \\
        &\le -2  \frac{\var_{\mu_{\Lambda}}(P_t\varphi)^p}{A^p\osc(\varphi)^{2p/q}}\, ,
\end{align*}
since $\osc(P_t \varphi) \le \osc(\varphi)$. Equivalently,
\begin{equation*}
\frac{\rmd}{\rmd t} \var_{\mu_{\Lambda}}(P_t\varphi)^{1-p} \ge \frac{2(p-1)}{A^p\osc(\varphi)^{2p/q}}\,.
\end{equation*}
By integrating we find for all $t>0$, 
\begin{align}
\var_{\mu_{\Lambda}}(P_t\varphi) &\le \frac{\var_{\mu_{\Lambda}}(\varphi)}{\Big(1+2(p-1)t\var_{\mu_{\Lambda}}(\varphi)^{p-1}/A^p\osc(\varphi)^{2p/q}\Big)^{1/(p-1)}} \label{eq:w_more_precise}\\
&\le \frac{A^q\osc(\varphi)^2}{(2(p-1)t)^{1/(p-1)}} \, .\nonumber 
\end{align}
Recalling $A=K \delta^{-1/q} N^{\kappa}$ this yields Eq.~\eqref{eq:algebraic} of Theorem~\ref{thm:algebraic_decay}.
\end{proof}

\paragraph{Mixing from a warm start} 
Next, as a preliminary to the construction of the sampling algorithm we first show that the weak PI~\eqref{eq:wPI_nu2} implies polynomial time mixing from a \emph{warm start}. 
 Let us apply Eq.~\eqref{eq:wPI_nu2} to the indicator function $\varphi = \one_{S}$, $S \subset \{-1,+1\}^{\Lambda}$:
\begin{equation}\label{eq:large_set_expansion}
\mu_{\Lambda}(S)(1-\mu_{\Lambda}(S)) \le A Q(S,S^c)^{1/p} \, ,
\end{equation}
where $Q(S,S^c) = \sum_{x\in S,y\in S^c} \mu_{\Lambda}(x)P(x,y)$, and $P$ is the transition matrix of Glauber dynamics. Note that the inequality~\eqref{eq:large_set_expansion} is a statement of large set expansion: it is a conductance bound which is most effective when $\mu_{\Lambda}(S)$ is neither too small nor too close to 1.

For ease of notation in this subsection, let $(\sigma_k)_{k \ge 0}$ be the discrete-time Markov chain of Glauber dynamics, and let $\pi_k$ be the distribution of $\sigma_k$ with $\pi_{\infty} = \mu_{\Lambda}$. 

\begin{Lem}\label{lem:warm-start-mixing}
Suppose $\pi_{\infty}$ satisfies the weak conductance bound
\begin{equation}\label{eq:weak_conductance}
\pi_{\infty}(S)(1-\pi_{\infty}(S)) \le  A Q(S,S^c)^{1/p} \, ,~~~ \forall ~S \subseteq \{-1,+1\}^{\Lambda}\, ,
\end{equation}
with $p \ge 1$, $A^p\ge 2/4^{p-1}$, and we further assume a warm start condition for the chain: 
\begin{equation}\label{eq:warm_start}
\max_{S \subseteq \{-1,+1\}^\Lambda} \frac{\pi_0(S)}{\pi_{\infty}(S)} \le M\, , 
\end{equation}
 for some $M>0$. Then 
\begin{equation}\label{eq:tv}
d_{\sTV}\big(\pi_k , \pi_{\infty}\big) \le M \left(\frac{A^{2p}\log k}{k}\right)^{1/(2p-1)} \, ,~~~ k \ge 1 \, .
\end{equation}
\end{Lem}
Before starting the proof let us remark that in our application of the above Lemma, $A = K \delta^{-1/q}N^{\kappa}$ for which~\eqref{eq:large_set_expansion} holds from~\eqref{eq:wPI_nu2}. We may assume that the condition imposed on $A$ is satisfied provided $N$ is large enough or $\delta$ is small enough.
\begin{proof}
We follow an approach of Lovasz and Simonovits~\cite{lovasz1993random}. We consider the function   
\begin{equation*}
h_{k}(x) = \sup_{g} \int_{\{-1,+1\}^{\Lambda}} g \rmd \pi_k - x \, ,~~~ x \in [0,1]\, ,
\end{equation*}
where the supremum is taken with respect to all measurable functions $g : \{-1,+1\}^{\Lambda} \to [0,1]$ such that $\int g \rmd \pi_{\infty} = x$. We observe that
\begin{equation}\label{eq:tv_h}
d_{\sTV}\big(\pi_k , \pi_{\infty}\big) = \max_{x} h_k(x)\, .
\end{equation}
A straightforward adaptation of Lemma 1.3  in~\cite{lovasz1993random} which uses the bound~\eqref{eq:weak_conductance} states the following recursion:   
\begin{equation}\label{eq:recursion}
h_{k}(x) \le \frac{1}{2}\Big[h_{k-1}\big(x- 2(A^{-1}x(1-x))^{p}\big) + h_{k-1}\big(x + 2(A^{-1}x(1-x))^{p}\big)\Big] \, ,~~~  k \ge 1\, ,
\end{equation} 
valid for all $x$ such that the arguments in the above display are in $[0,1]$. A sufficient condition for this to hold for all $x\in [0,1]$ is $A^p \ge 2/4^{p-1}$.
By Eq.~\eqref{eq:warm_start} we have $h_{0}(x) \le Mx $. Since we also have $h_{0}(x) \le 1-x$ it follows that 
\begin{equation*}
h_{0}(x) \le \sqrt{Mx(1-x)} \, .
\end{equation*}
Using this we have 
\begin{align*}
h_{0}\big(x- 2(A^{-1}x(1-x))^{p}\big)^2 &\le M x(1-x)\big(1-2A^{-p}x^{p-1}(1-x)^{p}\big)\big(1+2A^{-p}x^{p}(1-x)^{p-1}\big)\,, \\
h_{0}\big(x+ 2(A^{-1}x(1-x))^{p}\big)^2 &\le M x(1-x)\big(1+2A^{-p}x^{p-1}(1-x)^{p}\big)\big(1-2A^{-p}x^{p}(1-x)^{p-1}\big)\, .
\end{align*}
With $a= 2A^{-p}x^{p-1}(1-x)^{p}$ and $b = 2A^{-p}x^{p}(1-x)^{p-1}$, the recursion~\eqref{eq:recursion} implies the following bound on $h_1(x)$:
\begin{align*}
h_{1}(x) &\le \frac{1}{2}\sqrt{Mx(1-x)} \Big( \sqrt{(1-a)(1+b)} + \sqrt{(1+a)(1-b)}\Big) \, ,\\
&\le \sqrt{Mx(1-x)} \sqrt{1-ab} \, ,\\
&=\sqrt{Mx(1-x)} \big(1-4A^{-2p}x^{2p-1}(1-x)^{2p-1}\big)^{1/2}\, ,
\end{align*}
where the second line follows from Jensen's inequality. 
Now assume that $x(1-x) \ge r$ for some $r>0$. Then 
\begin{equation*}
h_{1}(x) \le \sqrt{Mx(1-x)} \big(1-4A^{-2p}r^{2p-1}\big)^{1/2} \, .
\end{equation*}
By iterating this bound, if follows that 
\begin{align*}
h_{k}(x) &\le \sqrt{Mx(1-x)} \big(1-4A^{-2p}r^{2p-1}\big)^{k/2} \, , \\
&\le \sqrt{M} e^{-2A^{-2p}r^{2p-1}k}\, ,~~~ k \ge 0 \, .
\end{align*}
On the other hand we have the bound $h_{k}(x) \le h_0(x) \le Mx$ for all $x$; indeed  
we have $h_{k+1}(x) = \sup_{g} \int Pg \rmd \pi_k - x \le h_k(x)$ since $Pg \in [0,1]$ and $\int Pg \rmd \pi_{\infty} = x$ since $\pi_{\infty}$ is stationary for $P$. 
Now from Eq.~\eqref{eq:tv_h} it follows that for any $r \in (0,1)$,
\begin{align*}
d_{\sTV}\big(\pi_k , \pi_{\infty}\big) &\le  \max\Big\{\sqrt{M} e^{-2A^{-2p}r^{2p-1}k} \, , \, Mr\Big\} \\
&\le \max\left\{\frac{\sqrt{M}}{k} \, , \, M \left(\frac{A^{2p}\log k}{2k}\right)^{1/(2p-1)} \right\}\, ,
\end{align*}
where the last line follows by taking $r^{2p-1} = (A^{2p}\log k)/(2k) $. 
\end{proof}

\subsubsection{A sampling algorithm under WSM}

In this subsection we prove Theorem~\ref{thm:sampling_alg}. We use the mixing result from a warm start of the previous subsection to build a polynomial-time sampler from the $\RFIM$ on $\Lambda_n$.

This sampler will proceed by building the sample up one vertex at a time. 
Fix an enumeration $v_1=o,\cdots,v_N$ of the vertices of $\Lambda_n$ in such a way that  $\Lambda^{(i)} = \{v_j\in \Lambda: j\le i\} \in \mathcal{Q}_n$ (recall that $\mathcal{Q}_n $ is the collection of `cube-like' subsets $B_r(o) \cup (\partial B_r(o)\cap A)$, with $r \le n-1$ and $A \subset \ZZ^d$.) 
This can be done by having the enumeration grow layer by layer in a nucleation fashion, and such that $\Lambda^{(i)}$ is always a union of at most $d$ cubes. 
The algorithm is as follows:
\begin{enumerate}
    \item Start with a perfect sample $X_*^{(1)}$ of the $\RFIM$ on a single origin vertex $\Lambda^{(1)} = \{v_1\}$. ($X_*^{(1)}\in\{-1,+1\}$ is drawn with probabilities proportional to $e^{\pm h_{v_1}}$.) 
    \item For each $i \ge 2$, consider a Glauber dynamics chain $(X^{(i)}_k)_{k \ge 0}$ on the $\RFIM$ measure $\mu_{\Lambda^{(i)}}$ on $\Lambda^{(i)}$  initialized from $X_0^{(i)} = [X_*^{(i-1)},\sigma_{v_i}] \in \{-1,+1\}^{\Lambda^{(i)}}$, the concatenation of $X_*^{(i-1)}$ with a spin $\sigma_{v_i}\in \{\pm 1\}$ drawn independently with probabilities proportional to $e^{\pm h_{v_i}}$. 
    \item Run Glauber dynamics from this initialization for time $k_* = N^{C_*}$ to obtain $X_*^{(i)} = X^{(i)}_{k_*}$.
\end{enumerate}
The constant $C_*$ will be chosen subsequently, depending on the total variation distance $\eps$ and probability $\delta$ we are aiming for.

\begin{Prop}\label{prop:sampling-algorithm}
    The sampling algorithm described by (1)--(3) above produces a sample that is $\eps$-close in total variation to $\mu_{\Lambda_n}$ with $\P_h$-probability $1-\delta$ and in time that is polynomial in $N$. This can be achieved for any $\eps, \delta \gtrsim 1/\textup{poly}(N)$. 
\end{Prop}

\begin{proof}
Fix $\Lambda_n$ as our target domain and $\eps>0$ a target total-variation distance. Our goal is to establish that as long as $C_*$ is a large enough constant, the sampler described by (1)--(3) above yields 
\begin{align}\label{eq:sampler-inductive-aim}
d_{\sTV}\big(\pi^{(i)}_{k_*}\,,\, \mu_{\Lambda^{(i)}}\big) \le \frac{i \eps}{N}\, ,~~~\mbox{for all } 1\le i \le N \, ,
\end{align}
where $\pi^{(i)}_{k}$ is the distribution of $X^{(i)}_{k}$ and $k_* = N^{C_*}$ as described in the algorithm. 
Taking the final $i=N$ so that $\Lambda^{(i)} = \Lambda_n$, we obtain $d_{\sTV}(\pi^{(N)}_{k_*},\mu_{\Lambda_n}) \le \eps$. We prove this bound inductively. For the base case, we are using a perfect sample from $\mu_{\Lambda^{(1)}}$ so the bound trivially holds. 

Let us assume the bound holds for $i-1$. Then $\pi_{0}^{(i)}$ is within total variation distance $(i-1)\eps/N$ from $\mu_{\Lambda^{(i-1)}}\otimes \mu_{v_i}$, where $\mu_{v_i}(\sigma) \propto e^{ h_{v_i} \sigma}$. Coupling the Glauber dynamics realization $X^{(i)}_k$ initialized from $\pi_0^{(i)}$ with $Y^{(i)}_k$ initialized from $\mu_{\Lambda^{(i-1)}}\otimes \mu_{v_i}$ by using the optimal total variation coupling on the initializations $(X^{(i)}_0,Y^{(i)}_0)$, and then coupling the dynamics with the identity coupling if $X^{(i)}_0 = Y^{(i)}_0$, we obtain
\begin{align}\label{eq:inductive-tv-bound-sampler}
    d_{\sTV}\big(X^{(i)}_k,\mu_{\Lambda^{(i)}}\big) \le d_{\sTV}\big(X_{0}^{(i)},Y_{0}^{(i)}\big) + d_{\sTV}\big(Y_{k}^{(i)},\mu_{\Lambda^{(i)}}\big) \le (i-1)\eps/N + d_{\sTV}\big(Y_{k}^{(i)},\mu_{\Lambda^{(i)}}\big) \, .
\end{align}
The distribution of $Y^{(i)}_0$, $\mu_{\Lambda^{(i-1)}}\otimes \mu_{v_i}$, has Radon--Nikodym derivative to $\mu_{\Lambda^{(i)}}$ of at most $M=e^{4d\beta}$ and thus satisfies the warm start condition~\eqref{eq:warm_start} of Lemma~\ref{lem:warm-start-mixing} since the maximal change in energy from adding interactions between $v_i$ and $\Lambda^{(i-1)}$ is $4d\beta$. Therefore Lemma~\ref{lem:warm-start-mixing} implies 
\begin{align}\label{eq:tvY}
    d_{\sTV}\big(Y^{(i)}_{k}, \mu_{\Lambda^{(i)}}\big) \le M\Big(\frac{A_i^{2p}\log k}{ k}\Big)^{1/(2p-1)} \, ,
\end{align}
with probability $1-\delta$, where $A_i = K \delta^{-1/q} |\Lambda^{(i)}|^{\kappa}$ as per Theorem~\ref{thm:weak_PI} applied to $\Lambda^{(i)}$. By replacing $\delta$ by $\delta/N$ and applying a union bound, the bound in Eq.~\eqref{eq:tvY} holds uniformly over $i$ with probability at least $1-\delta$. By taking $k=N^{C_*}$ for $C_*$ sufficiently large, we obtain $\max_{1 \le i\le N}\, d_{\sTV}(Y^{(i)}_{k} , \mu_{\Lambda^{(i)}})\le \eps/N$ under this event. Combining with~\eqref{eq:inductive-tv-bound-sampler}, we get the claimed bound of~\eqref{eq:sampler-inductive-aim}. 
\end{proof}

\section{Bound on the operator norm of the covariance matrix}
\label{sec:covariance}
This section is devoted to the proof of a tail bound on the operator norm of the covariance matrix of $\nu_t$, Theorem~\ref{thm:trace}. Fix $\Lambda \subset \ZZ^d$ with volume $|\Lambda|=N$.  Recall the notation $A_t = \cov(\nu_t)$ where $(\nu_t)_{t \ge 0}$ is the SL process started at $\nu_0 = \mu_{\Lambda}$; see Section~\ref{sec:preliminaries}.

We will rely on the observation that the correlation between the spins at $u$ and $v$ can be controlled by introducing a boundary separating the two vertices. For a vertex $u \in \Lambda$ and an integer $\ell\ge 0$ we define 
 \begin{equation}\label{eq:def_delta}
\delta_t(u, \ell) := d_{\sTV}\Big(\nu_t\big(\sigma_{u} \in \cdot \,|\, \sigma_{\partial B_{\ell}(u)} = +\big) \,,\, \nu_t\big(\sigma_{u} \in \cdot \,|\, \sigma_{\partial B_{\ell}(u)} = -\big)\Big)\, ,
 \end{equation}
 where we recall that $B_{\ell}(u)$ is the $\ell_{\infty}$-ball of radius $\ell$ around $u$ in $\Lambda$.
 Let us define $\langle \,\cdot\, \rangle_{t}$ to be the Gibbs average with respect to $\nu_t$ and 
\begin{equation*} 
\langle \sigma_u ;\sigma_v \rangle _t := (\cov(\nu_t))_{u,v} = \langle \sigma_u \sigma_v \rangle_t - \langle \sigma_u \rangle _t\langle \sigma_v \rangle_t \,,~~~ u,v \in \Lambda\, .
\end{equation*}

\begin{Lem}\label{lem:cov_inf}
For all $u , v \in \Lambda$, $t\ge0$ and integers  $\ell \le d_{\infty}(u,v)$ we have
\begin{equation*}
0 \le \langle \sigma_u ; \sigma_v\rangle_t \le \delta_t(u,\ell) \, .
\end{equation*}
\end{Lem}

\begin{proof}
The proof is standard and follows from monotonicity of the magnetization of $\sigma_u$ with respect to the boundary condition (a consequence of the $\FKG$ inequality) and the fact that $\partial B_{\ell}(u)$ separates $u$ from $v$ in $\Lambda$. See e.g., \cite[Corollary 1.3]{ding2022new} for details. 
\end{proof}

It follows from Lemma~\ref{lem:cov_inf} that
\begin{equation}    \label{eq:trace_TV}
    \tr(A_t^p) = \sum_{u_1, \dots, u_p \in \Lambda} \prod_{i=1}^p \, \langle \sigma_{u_i};\sigma_{u_{i-1}}\rangle _{t}
    \leq \sum_{u_1, \dots, u_p \in \Lambda_n} \prod_{i=1}^p \, \delta_t(u_i,\ell_i)\, ,
\end{equation}
where $u_{0}=u_1$ and $\ell_i \le d_{\infty}(u_i,u_{i-1})$.
Next we control the expectation of the right-hand side in Eq.~\eqref{eq:trace_TV} uniformly in $t$ via the next proposition:
\begin{Prop}\label{prop:bd_product}
Fix $p \ge 1$, $t\ge 0$, a subset $A \subseteq [p]$ and a $p$-tuple of vertices $u_1,\cdots,u_p \in \Lambda$.  
For any sequence of integers $(\ell_i)_{i \in A}$ such that
 \begin{align}\label{eq:cond_elli}
 d_{\infty}(u_i,u_{i-1}) \ge \ell_i \,, ~~~ i  \in A \, ,
 ~~~\mbox{and}~~~~ 
 d_{\infty}(u_i,u_{j}) \ge 2(\ell_i+\ell_j) \,,~~~ i \neq j \in A \, ,
  \end{align}
we have 
   \begin{equation}\label{eq:bd_product}
 \EE\Big[\,\prod_{i \in A} \delta_t(u_i,\ell_i) \, \Big| \, h\Big] \le \prod_{i \in A} \Big(\sum_{v \in B_{\ell_i}(u_i)} (1+\one_{v=u_i})\, \delta_0(v,\ell_i)\Big)\, .
 \end{equation}
\end{Prop}

It follows from Eq.~\eqref{eq:trace_TV},  the fact $0 \le \delta_t(u_i,\ell_i) \le 1$, and Proposition~\ref{prop:bd_product} that 
 \begin{equation}\label{eq:bd_trace_gamma}
 \sup_{t\ge 0} \, \EE \big[\tr(A_t^p) \,\big|\, h \big] \le \Gamma(A,L) \, ,
 \end{equation}
 where the quantity in the right-hand side is defined as follows:
For a map $A : (u_1,\cdots,u_p) \in \Lambda_n^p \mapsto A(u_1,\cdots,u_p) \subseteq [p]$ from the set of $p$-tuples of vertices to the subsets of $[p]$, and a map $L : (u_1,\cdots,u_p) \in \Lambda_n^p \mapsto (\ell_1,\cdots,\ell_p) \in \mathbb{N}^p$, $\ell_i = (L(u_1,\cdots,u_p))_i$, $i \in [p]$, 
we let 
 \begin{equation}\label{eq:def_gamma}
\Gamma(A,L) := \sum_{u_1, \dots, u_p \in \Lambda} \, \prod_{i \in A(u_1,\cdots,u_p)} \Big(\sum_{v \in B_{\ell_i}(u_i)} (1+\one_{v=u_i})\, \delta_0(v,\ell_i)\Big)\, .
 \end{equation}
 
  Next we observe that the collections $(\delta_0(v,\ell_i))_{v \in B_{\ell_i}(u_i)}$ and $(\delta_0(v,\ell_j))_{v \in B_{\ell_j}(u_j)}$ are independent for $i \neq j \in A(u_1,\cdots,u_p)$. This is true since these collections depend on the fields $(h_{x})_{x\in B(u_k,2\ell_k)}$ for $k = i,j$ respectively, and the balls $B(u_i,2\ell_i)$ are disjoint due to~\eqref{eq:cond_elli}. By $\WSM$ the term inside the parentheses in~\eqref{eq:def_gamma} is bounded by $C\ell_i^d e^{-c\ell_i}$ in expectation. Then the combinatorial sum is dealt with via a judicious choice of $A$ and the distances $\ell_i$:
 
\begin{Prop}\label{prop:bound_gamma} 
 Suppose the $\RFIM$ on $\Lambda$ satisfies $\WSM(C)$ for some $C>0$. There exists two maps $A$ and $L$ as defined above and satisfying the conditions of Eq.~\eqref{eq:cond_elli}, and a constant $C_0 = C_0(d,C)$ such that 
 \begin{equation}\label{eq:bound_gamma}
\E_h[\Gamma(A,L)] \le C_0^p \, p! \, N \, .
 \end{equation}
 \end{Prop}

Theorem~\ref{thm:trace} follows immediately from Eq.~\eqref{eq:bd_trace_gamma} and Proposition~\ref{prop:bound_gamma}.
It remains to prove Proposition~\ref{prop:bd_product} and Proposition~\ref{prop:bound_gamma}.
Behind the construction leading to the bound~\eqref{eq:bound_gamma} is a desire to optimize the tradeoff between the number of terms $\delta_t(u_i,\ell_i)$ we keep in the product appearing in~\eqref{eq:trace_TV} and how large we are allowed to take the radii $\ell_i$. The larger the $\ell_i$'s, the smaller these individual terms are in expectation, but the fewer of them we can have in a finite domain $\Lambda$ because of the separation constraints~\eqref{eq:cond_elli}. 
We detail a specific construction in the proof of Proposition~\ref{prop:bound_gamma} which mimics a peeling procedure: Each vertex $(u_i)_{i\ge 2}$ contributes by a constant to the overall sum~\eqref{eq:trace_TV} due to $\WSM$, leading to the term $C_0^p$ in the final bound, and $u_1$ contributes by $|\In|=N$. The term $p!$ arises from all possible orderings in this peeling.

Next, in the proof of Proposition~\ref{prop:bd_product} we crucially rely on a monotonicity property with respect to $t$ of the expected total variations $\E[\delta_t(u,\ell)\,|\, h]$. More precisely, the process $t \mapsto \delta_t(u,\ell)$ is a super-martingale. As we will see, this is a consequence of the $\FKG$ inequality. We state and prove the result in a slightly more general setting:    

Let $G=(V,E)$ be a finite graph, and fix an arbitrary external field $h=(h_u)_{u\in V}$. We let
  \begin{equation*}
  \mu(\sigma) \propto \exp \Big( \sum_{(u,v) \in E} J_{uv} \sigma_u \sigma_v + \sum_{u \in V} h_u \sigma_u \Big) \, ,  
  \end{equation*}
  where $J_{uv} \ge 0$ for all $(u,v) \in E$. We make the dependence of the external field explicit by writing $\langle \, \cdot \, \rangle_{h}$ as the Gibbs average when the external field vector is $h$.
 Note that the Gibbs average with respect to the corresponding stochastically localized measure is $\langle \, \cdot \, \rangle_{h+y_t}$. 
 
\begin{Lem}\label{lem:decreasingcorr}
Consider an auxiliary field $g \in \R^V_+$ with nonnegative entries, and consider averages of the form 
 \begin{equation}
 X^{\pm}_t := \langle f(\sigma)\rangle_{y_t + h \pm g} \, ,
  \end{equation}
where $f : \{-1,+1\}^V \to \R$ is an increasing function.
Then $X^+$ is a super-martingale and $X^-$ is a sub-martingale.
 
 In particular, letting $f(\sigma) = \sigma_u$ and $g_u = +\infty$ on the boundary vertices and $g_u = 0$ on the remaining vertices, we obtain
  \begin{equation} \label{eq:decreasingcorr}
  \E  \big[\delta_t(u,\ell) \big] \le \delta_0(u,\ell) \, .
  \end{equation}
\end{Lem} 

\begin{proof}
 We let 
 \begin{align*}
 \bar{w}_t(\sigma) &=  \exp\Big(\sum_{(u,v) \in E} J_{uv} \sigma_u \sigma_v + \sum_{u \in V} (y_t + h + g)_u \, \sigma_u \Big) \, ,\\ 
 Z_t &= \sum_{\sigma \in \{\pm 1\}^V} \bar{w}_t(\sigma) \, ,~~~  \text{and}~~~ w_t(\sigma) = \frac{1}{Z_t} \bar{w}_t(\sigma) \, . 
 \end{align*}
 By It\^{o}'s formula we have
 \[\rmd \bar{w}_t(\sigma) = \bar{w}_t(\sigma)  \Big( \sigma^\top  \rmd y_t + \frac{1}{2}\|\sigma\|^2 \rmd t \Big) \, , \]
 and 
 \begin{align*}
\rmd w_t(\sigma) &=  \frac{\rmd \bar{w}_t(\sigma)}{Z_t} + \bar{w}_t(\sigma) \rmd Z_t^{-1} +  \rmd [\bar{w}_t, Z_t^{-1}] \, ,\\
\rmd Z_t^{-1} &= - \frac{\rmd Z_t}{Z_t^2} + \frac{\rmd [Z_t]}{Z_t^3} \, , ~~~\mbox{and}~~~
\rmd [\bar{w}_t, Z_t^{-1}] = - \frac{\rmd [\bar{w}_t,Z_t]}{Z_t^2} \, .
\end{align*}
Therefore,
\begin{align*}
\hspace{-.5cm} \rmd w_t(\sigma) 
&=  w_t(\sigma)  
\Big( \sigma^\top  \rmd y_t + \frac{n}{2} \rmd t - \sum_{\tau} w_t(\tau)(\tau^\top  \rmd y_t + \frac{n}{2} \rmd t ) \\
& \hspace{7cm}
 +  \Big\|\sum_{\tau} w_t(\tau)  \tau \Big\|^2 \rmd t
-  \sum_{\tau} w_t(\tau)  \tau^\top \sigma \rmd t\Big) \nonumber\\
&=  w_t(\sigma)  \Big( \sigma - \sum_{\tau} w_t(\tau)\tau \Big)^\top \Big(\rmd y_t - \sum_{\tau} w_t(\tau)\tau \rmd t\Big)\nonumber\\
&=  w_t(\sigma)  \Big( \sigma - \langle \sigma \rangle_{y_t + h + g} \Big)^\top \Big(\langle \sigma \rangle_{y_t + h} - \langle \sigma \rangle_{y_t + h + g} \Big)\rmd t +  w_t(\sigma)  \Big( \sigma - \sum_{\tau} w_t(\tau)\tau \Big)^\top \rmd B_t \, ,
\end{align*}     
where the last line plugged in~\eqref{eq:sde_sl} for $\dd y_t$. 
So we obtain
\begin{equation*}
\rmd X^+_t =  \Big( \langle f(\sigma) ; \sigma \rangle_{y_t + h + g} \Big)^\top \Big(\langle \sigma \rangle_{y_t + h} - \langle \sigma \rangle_{y_t + h + g} \Big)\rmd t + b_t^\top\rmd B_t \, ,
\end{equation*}
where we recall that $ \langle f ; g\rangle  =  \langle f g\rangle -  \langle f \rangle  \langle g\rangle $, and $b_t$ is some random vector in $\R^V$.

By the $\FKG$ inequality, since $g$ has nonnegative entries, $\langle \sigma \rangle_{y_t + h} \le  \langle \sigma \rangle_{y_t + h + g}$  entrywise and $\langle f(\sigma) ; \sigma \rangle_{y_t + h + g} \ge 0$ entrywise since $f$ is increasing. We conclude that the drift term in the above equation is nonpositive and therefore that $X^+_t$ is a super-martingale. 

Replacing $g$ by $-g$ in the above does not affect the positivity of the correlation $\langle f(\sigma) ; \sigma \rangle_{y_t + h - g}$, but the vector $\langle \sigma \rangle_{y_t + h} - \langle \sigma \rangle_{y_t + h - g}$ is now nonnegative. Therefore $X^-_t$ is a sub-martingale. 
\end{proof}

Before turning to the proof of Proposition~\ref{prop:bd_product} we first state the following  simple lemma.

\begin{Lem}\label{lem:settoboundary}
 Consider two subsets of vertices $\Delta \subset \Delta' \subseteq V$, 
 \begin{align}
  d_{\sTV} \big( \mu(\sigma_\Delta \in \, \cdot \, | \, \sigma_{\partial \Delta'} =+) , &  \mu(\sigma_\Delta \in \, \cdot \, | \, \sigma_{\partial \Delta'} =-)  \big) \nonumber\\
  & \le \sum_{v \in \Delta}  d_{\sTV} \big( \mu(\sigma_v \in \, \cdot \, | \, \sigma_{\partial \Delta'} =+) , \mu(\sigma_v \in \, \cdot \, | \, \sigma_{\partial \Delta'} =-)  \big) \label{eq:settoboundary} \, .
 \end{align}
\end{Lem} 

 \begin{proof}
 The upper bound follows immediately from the monotone coupling of the two measures on the left-hand side and a union bound over the vertices of $\Delta$. 
  \end{proof}

\begin{proof}[Proof of Proposition~\ref{prop:bd_product}]
    Fix $p\ge 1$, $t\ge 0$, a subset $A \subseteq [p]$ and a $p$-tuple of vertices $u_1,\cdots,u_p \in \Lambda$. Fix a sequence of integers satisfying condition Eq.~\eqref{eq:cond_elli}. 
    Without loss of generality we assume $A = \{u_1,\cdots,u_q\}$, $q\le p$. 
    We peel the terms in the product $\prod_{i \in A} \delta_t(u_i,\ell_i)$ one by one. We observe that the random variable $\delta_t(u_i,\ell_i)$ is a measurable function of $\{(y_t+h)_v : v \in B_{\ell_i}(u_i)\}$. Furthermore as per Proposition~\ref{prop:sl_bayes}, $y_t$ has the distribution of $t\sigma^* + B_t$ where $\sigma^* \sim \nu_0$ and $(B_t)$ is an independent standard Brownian motion. We define 
    \begin{equation*}
        X_i = \EE\big[\delta_t(u_i,\ell_i) \, \big| \, \sigma^*, h\big] \, .    
    \end{equation*}
     (The expectation is with respect to $B_t$.) This is a measurable function of $\big\{\sigma^*_v, h_v : v \in B_{\ell_i}(u_i)\big\}$. 
    Since the balls $B_{\ell_i}(u_i)$ do not intersect, we have 
    \begin{equation*}
    \EE\Big[\,\prod_{i=1}^q \delta_t(u_i,\ell_i) \, \big| \, h\Big] = \EE\Big[\,\prod_{i=1}^q X_i \, \big| \,  h\Big]
    = \nu_0\Big(\prod_{i=1}^q X_i \Big) \, .
    \end{equation*}    
    Moreover, since $d(u_1,u_i) \ge 2\ell_1 + \ell_i$,  $B_{2\ell_1}(u_1)$ and $B_{\ell_i}(u_i)$ are disjoint for all $i$, therefore 
    \begin{align}\label{eq:sepration}
     &\nu_0\Big(\prod_{i=1}^q X_i \Big) - \nu_0\big(X_1\big) \cdot \nu_0\Big(\prod_{i=2}^q X_i \Big) \nonumber\\
     &~~~=\nu_0\Big( \Big(\nu_0\big(X_1 | \sigma_{\partial B_{2\ell_1}(u_1)}^*\big) - \nu_0(X_1)\Big) \cdot \nu_0\Big(\prod_{i=2}^q X_i \,|\, \sigma_{\partial B_{2\ell_1}(u_1)}^*\Big)\Big)\nonumber \\
     &~~~\le \sup_{\tau,\tau'} \Big(\nu_0\big(X_1 | \sigma_{\partial B_{2\ell_1}(u_1)}^* = \tau\big) -\nu_0\big(X_1 | \sigma_{\partial B_{2\ell_1}(u_1)}^* = \tau'\big)\Big) \cdot \nu_0\Big(\prod_{i=2}^q X_i \Big)\, .
    \end{align}
    Since $0 \le X_1\le 1$ and by monotonicity with respect to the boundary condition, the above supremum is bounded above by 
    \begin{align}\label{eq:tv_vertices}
    &d_{\sTV}\Big(\nu_0\big(\sigma_{B(u_1,\ell_1)}^* \in \cdot \,|\, \sigma_{\partial B_{2\ell_1}(u_1)}^* = +\big),  \nu_0\big(\sigma_{B(u_1,\ell_1)}^* \in \cdot \,|\, \sigma_{\partial B_{2\ell_1}(u_1)}^* = -\big) \Big)\nonumber\\
    &~~~\le \sum_{v \in B(u_1,\ell_1)}  d_{\sTV}\Big(\nu_0\big(\sigma_{v}^* \in \cdot \,|\, \sigma_{\partial B_{2\ell_1}(u_1)}^* = +\big), \nu_0\big(\sigma_{v}^* \in \cdot \,|\, \sigma_{\partial B_{2\ell_1}(u_1)}^* = -\big) \Big) \nonumber\\
    &~~~\le  \sum_{v \in B(u_1,\ell_1)} \delta_0(v,\ell_1) \, .
    \end{align}
    The second line follows from Lemma~\ref{lem:settoboundary}, and the last line from the observation that $B_{\ell_1}(v) \subseteq B_{2\ell_1}(u_1)$.
It follows from Eq.~\eqref{eq:sepration} and Eq.~\eqref{eq:tv_vertices} that
    \begin{equation*}
    \nu_0\Big(\prod_{i=1}^q X_i \Big) \le \prod_{i=1}^q \Big(\nu_0\big(X_i\big) + \sum_{v \in B_{\ell_i}(u_i)} \delta_0(v,\ell_i) \Big) \, .
    \end{equation*}
    We conclude the proof by noting that $\nu_0(X_i) = \E\big[\delta_t(u_i,\ell_i)\,|\, h\big] \le \delta_0(u_i,\ell_i)$ by Lemma~\ref{lem:decreasingcorr}. 
\end{proof}

 \begin{proof}[Proof of Proposition~\ref{prop:bound_gamma}]
 Assuming condition~\eqref{eq:cond_elli}, the balls $B_{2\ell_i}(u_i)$ are disjoint, and therefore
\begin{align*}
\E_h\big[\Gamma(A,L)\big] &= \sum_{u_1, \dots, u_p \in \Lambda} \, \prod_{i \in A(u_1,\cdots,u_p)} \Big(\sum_{v \in B_{\ell_i}(u_i)} (1+\one_{v=u_i})\, \E_h\big[\delta_0(v,\ell_i)\big]\Big)\\
&\le \sum_{u_1, \dots, u_p \in \Lambda} \, \prod_{i \in A(u_1,\cdots,u_p)} \big(2C\,\ell_i^{d} \,e^{-\ell_i/C}\big) \, ,
\end{align*}
where we used $\WSM(C)$ to obtain the last line. 
Next we build the maps $A$ and $L$.
For a sequence of points $u_1, \dots, u_p \in \Lambda$ and a set $I \subseteq \{1,\cdots,p\}$ we define 
\begin{equation} \label{eq:def_L} 
    \ell_i(I):= \frac{1}{4} \min_{j \in \{i-1\} \cup I \setminus \{i\}} d_{\infty}(u_i,u_j)\, , ~~~~ i \in [p]\, .
\end{equation}
Next we define for all $i \in [p]$,
\begin{align*} 
r_i = r_i(u_1,...,u_p) &:= \frac{1}{4} \min_{j < i} \,d_{\infty}(u_i,u_j) \, ,\\
\mbox{and}~~
J_i = J_i(u_1,...,u_p) &:= \text{argmin}_{j < i}\,\, d_{\infty}(u_i,u_j) \, .
\end{align*}
Now, define the set of indices that are relatively close
\begin{equation*} 
Q_k = Q_k(u_1,...,u_p) := \big\{i \in [p]\, : ~ r_i \in [2^{k}, 2^{k+1}-1]\big\} \, ,~~~ k \ge 0\, , 
\end{equation*}
and let 
\begin{equation*}
k_* = k_*(u_1,...u_p) := \mathrm{argmax}_k \, |Q_k| 2^k \, .
\end{equation*}
Finally, let 
\begin{equation}\label{eq:def_A} 
A(u_1,...,u_p) := Q_{k_*}\, ,~~~\mbox{and}~~~ \ell_i(u_1,...,u_p) := \ell_{i}(A(u_1,...,u_p))\, .
\end{equation}
With these definitions, conditions~\eqref{eq:cond_elli} are satisfied: $d_{\infty}(u_i,u_{i-1}) \ge \ell_i$ and $d_{\infty}(u_i,u_{j}) \ge 2(\ell_i+\ell_j)$ for all $i,j\in Q_{k_*}$, $i \neq j$. Next we claim that
\begin{equation} \label{eq:est_comb}
\sum_{u_1, \dots, u_p \in \Lambda_n} \prod_{i \in A} 2C\ell_i^d e^{-\ell_i/C} \leq C_0^p N p! \, ,
\end{equation}
where $C_0$ is only a function of $d$ and $C$, and we wrote $A$ and $\ell_i$ for brevity. 
The key idea about the above definitions is that 
\begin{equation}\label{eq:ell_half}
    \ell_i \geq \frac{1}{2} r_i\, , ~~ \forall i \in Q_{k_*}\, .
\end{equation}
Indeed, suppose $\ell _i <\frac{1}{2}r_i$. Then there exists $j \in Q_{k_*}$, $j>i$ such that $d_{\infty}(u_i,u_j)\le \frac{1}{2}r_i$. Since $j \in Q_{k_*}$, $r_j$ is on the order of $r_i$; in particular $r_j\geq\frac{1}{2}r_i$. But $r_j \leq d_{\infty}(u_i,u_j)$, which leads to a contradiction.
Therefore
\begin{align} \label{eq:eq2comb}
\prod_{i \in Q_{k_*}} 2C\ell_i^d e^{-\ell_i/C} \leq c^{p}\prod_{i \in Q_{k_*}}  e^{-\ell_i/2C} \leq c^p\,e^{- \sum_{i \in Q_{k_*}} r_i / 4C}\,,
\end{align}
where $c = (2C)^p\max_{x>0} x^d e^{-x/2C}$, and we used~\eqref{eq:ell_half} in the second inequality.
Now we claim that $\sum_{i \in Q_{k_*}} r_i \ge c_0 \sum_{i \in [p]} \sqrt{r_i}$ for some absolute constant $c_0>0$. Indeed,
\begin{align*}
    \sum_{i \in [p]} \sqrt{r_i} &= \sum_{k \ge 0} \sum_{i \in Q_k} \sqrt{r_i} 
    \le \sum_{k\ge 0} 2^{-k/2}\sum_{i \in Q_k}  r_i\\
    &\le \Big(\sum_{k\ge 0} 2^{-k/2}\Big) \, \max_{k} \, \sum_{i \in Q_k}  r_i \, ,
\end{align*}
and
\begin{align*}
\max_{k} \, \sum_{i \in Q_k}  r_i &\le \max_{k} 2^{k+1} |Q_k| 
=  2^{k_*+1} |Q_{k_*}| \le 2 \sum_{i \in Q_{k_*}}  r_i \, .
\end{align*}

Plugging this to the display~\eqref{eq:eq2comb} yields
\begin{equation}\label{eq:gamma_2}
\E\big[\Gamma(A,L)\big] \le C_0^p \sum_{u_1, \dots, u_p \in \Lambda} \, \prod_{i=1}^p e^{-c_0 \sqrt{r_i}} \, ,
\end{equation}
for some $C_0,c_0>0$ depending on $d,C$. 
Now, there are $p!$ options for the choice of $J_i$'s. Given the $J_i$'s and the different choices for $r_i$, there are at most $|\Lambda| \prod_{i \in p} r_i^{d-1}$ different configurations. (We determine the position of $u_1$ and then given $r_i$ and $J_i$, the point $u_i$ is on the boundary of $B_{r_i}(u_{J_i})$. Notice that by construction $J_i < i$ which means that those offsets determine the entire configuration.) All in all, the right-hand side of \eqref{eq:gamma_2} is bounded by
\begin{equation*}
C_0^p \,|\Lambda| \,p! \sum_{r_1,...,r_p \in \mathbb{N}} \prod_{i \in [p]} r_i^{d-1} e^{- c\sqrt{r_i}} = C_0^p \Big(\sum_{r=0}^\infty r^{d-1}e^{- c\sqrt{r}}\Big)^p p!\, |\Lambda| \, ,
\end{equation*}
concluding the proof.
\end{proof}

\section{Fast mixing with large variance fields or SSM}\label{sec:large_variance}
Our aim in this section is to show that in either of the following situations: 
\begin{enumerate}
\item a $\RFIM$ with sufficiently anti-concentrated external field, whose absolute values are independent (though the signs are not necessarily independent),
\item a $\RFIM$ satisfying $\SSM(C)$ in expectation (per Definition~\ref{assump:SSM-RFIM0}) with independent fields,
\end{enumerate}
has $N^{o(1)}$ mixing time and inverse spectral gap. The first of these proves Theorem~\ref{thm:fast-mixing-large-variance0}, thereby concluding the missing step from the proof of Theorem~\ref{thm:algebraic_decay}, and the second proves Theorem~\ref{mainthm:SSM}. 

These two proofs follow a very similar strategy based on coarse-graining the lattice into blocks where the field is good, and a sub-critical set of blocks where the field is bad, and we therefore present them together. In particular, most of the steps in the proofs are identical, and in the instances they are different, we present proofs for each situation.

The formal assumption for (1) above of a sufficiently anti-concentrated field with independent absolute values is the following: 
\begin{Ass}\label{ass:anticoncentrated-field}
    Let $K$ be such that 
    \begin{align}\label{eq:rho-parameter}
    \rho = \rho(\beta,d,K):= \frac{e^{2d\beta - K}}{e^{ 2d\beta -K} + e^{ - 2d\beta +K}}\,,
\end{align}
is less than $1/40d$, say. 
Suppose that the $\RFIM$ with external field $(h_x)_{x\in \Lambda}$ is such that $(|h_x|)_x$ are i.i.d.\ (their signs may be dependent), and such that 
    \begin{align*}
        \mathbb P_h(|h_x|\le K) < \tfrac{1}{40d}\,.
    \end{align*}
\end{Ass}

We will prove the following theorem that includes in it both Theorem~\ref{thm:fast-mixing-large-variance0} and Theorem~\ref{mainthm:SSM}. 

\begin{Thm}\label{thm:fast-mixing-large-variance}
    Suppose $\beta,d$ and $\mathbb P_h$ are such that one of Assumption~\ref{ass:anticoncentrated-field}, or Assumption $\SSM(C)$ per Definition~\ref{assump:SSM-RFIM0}, holds. There exist $\kappa = \kappa(d,\beta,C)>0$ and $C_{d,\beta}>0$ such that for every $L \ge \kappa \log n$, the $\RFIM$ Glauber dynamics on $\Lambda\in \mathcal Q_n$ satisfies 
    \begin{align*}
        \mathbb P_h\Big(\gap_{\Lambda}^{-1}\ge e^{C_{d,\beta} L^{\frac{d-1}{d}}(\log L)^{d-1}}\Big) \le e^{ - L}\,.
    \end{align*}
    Making the choice $L = \kappa \log n$, this gives that $\gap_{\Lambda_n}^{-1}\le n^{o(1)}$ except with probability $n^{-\kappa}$. 
\end{Thm}

For readability, we will prove this for $\Lambda= \Lambda_n$. It can easily be checked that the addition of some boundary vertices can be handled in the obvious manners.

\subsection{The coarse graining}
We begin with a coarse graining of the external field, to split it into good regions where the field is fairly large in a coarse sense (depending on which of Assumptions~\ref{ass:anticoncentrated-field} or~\ref{assump:SSM-RFIM0} are used) and exponential decay of correlations strongly holds, and regions where the field is not large enough for the exponential decay of correlations to hold. 

To coarse-grain $\Lambda_n$, we take a large $R$ (which we will later take to be $\approx  \log L$) and consider $\Lambda_n^{(R)} = \Lambda_n \cap R\mathbb Z^d$. For every vertex $v\in \Lambda_n^{(R)}$, let $B_{v} = B_{R}(v)$ be the $\ell^\infty$-ball of radius $R$ (side-length $2R$) in $\Lambda_n$ centered at $v$. 

The coarse-grained percolation process is an $|h|$-measurable, or $h$-measurable (depending on which assumption we are taking), percolation process on $\Lambda_n^{(R)}$ consisting of assignments of $\{\good,\bad\}$ to $\Lambda_n^{(R)}$. We refer to Figure~\ref{fig:coarse-graining} for a visualization.

\begin{figure}
    \centering
    \includegraphics[width=0.33\textwidth]{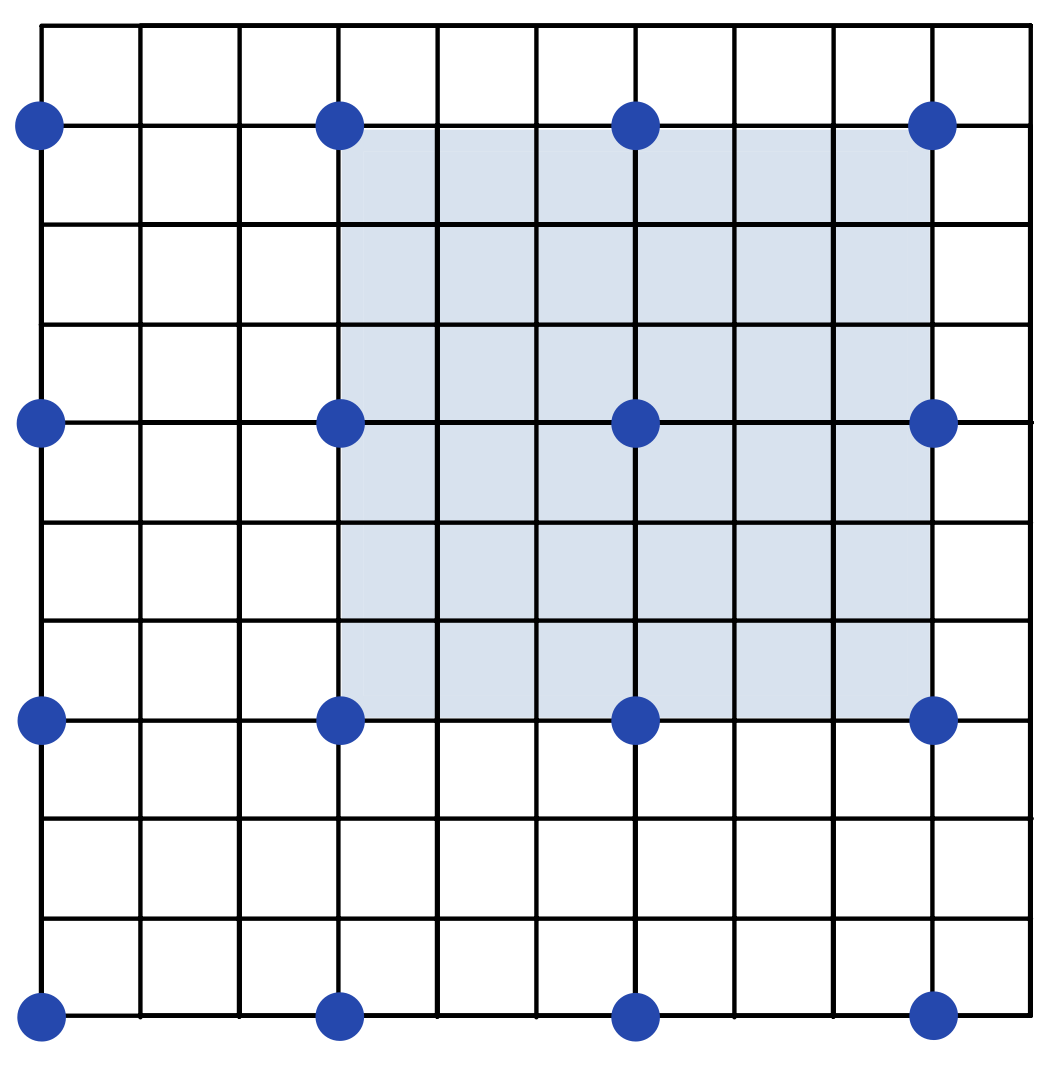}
\qquad \qquad\qquad 
\includegraphics[width = 0.33\textwidth]{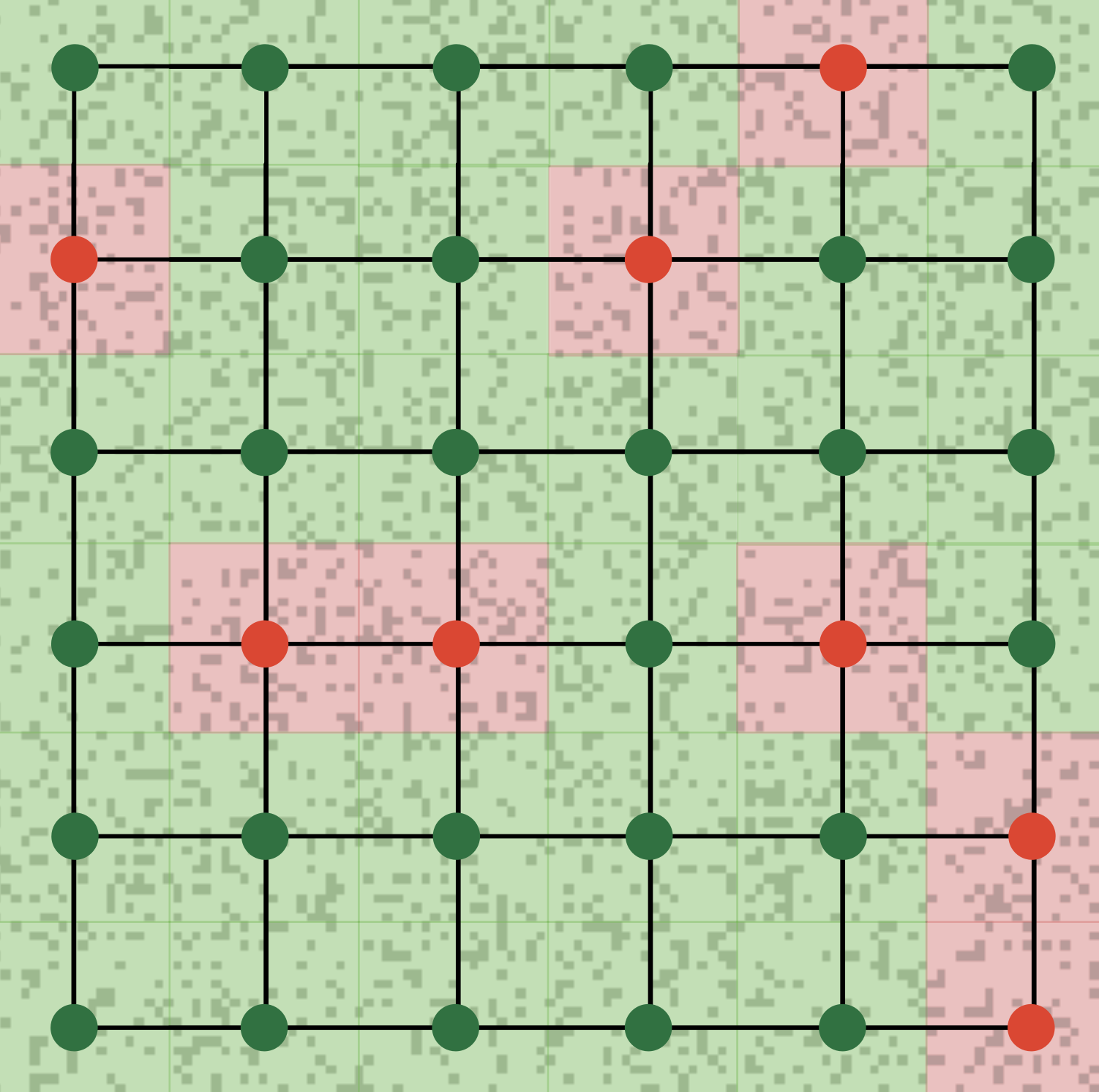}
    \caption{Left: the coarse-grained lattice $\Lambda_n^{(R)}$ (blue vertices) on top of $\Lambda_n$ (black); shaded blue is the $\ell^\infty$-ball $B_R(v)$ for a vertex $v\in \Lambda_n^{(R)}$. Right: a realization of the $\{\good,\bad\}$ coarse-grained process  (green and red vertices respectively) super-imposed on top of a percolation process on the original lattice $\Lambda_n$ marking regions with atypically small values of $h$.}
    \label{fig:coarse-graining}
    \vspace{-0.2in}
\end{figure}

\subsubsection{Good blocks of the coarse graining}
Here, we define the properties of $|h|$ or $h$ that make a block $v\in \Lambda_n^{(R)}$ be $\good$. There will be two different definitions for the two assumptions. 

When under Assumption~\ref{ass:anticoncentrated-field}, the coarse-grained percolation process is defined as follows. 

\begin{Def}\label{def:anticoncentrated-good-blocks}
    For any fixed realization of $|h|$ (and thus of $\omega^h$), let $\omega_x^h = \mathbf{1}_{|h_x|\le K}$ and let $\eta = (\eta_x)_{x\in \Lambda_n}$ be i.i.d.\ $\text{Ber}(\rho)$ from~\eqref{eq:rho-parameter}. Finally, let $\omega^\influence_x = \omega_x^h \vee \eta_x$ for each $x\in \Lambda_n$.
    
    A vertex $v\in \Lambda_n^{(R)}$ is called $\good$ if for every $w\in B_{v}$, for every $(\log R)^2 \le r \le R/8$,
        \begin{align*}
            \mathbb P(|\cC_w(\omega^\influence)|\ge r \mid \omega^h) \le e^{ - r}\,.
        \end{align*}
        where we use $\cC_w(\omega^\influence)$ for the connected component of $w$ in $\omega^\influence$. 
    If $v\in \Lambda_n^{(R)}$ is not $\good$, it is $\bad$. 
\end{Def}

When under Assumption~\ref{assump:SSM-RFIM0}, the coarse-grained percolation process is defined as follows. 

\begin{Def}\label{def:SSM-good-blocks}
    For any fixed realization of $h$, a vertex $v\in \Lambda_n^{(R)}$ is called $\good$ if for every $w\in B_v$, for every $(\log R)^2 \le \ell \le 2r \le R/4$,  
    \begin{align}\label{eq:SSM-good-event}
        \max_{B\in B_{r}(w) + \{-r+1,...,r-1\}^d} \max_{\substack{z\in \partial B \\ d(w,z) = \ell}} \max_{\eta \in \{\pm 1\}^{\partial B\setminus z}} d_{\sTV} \big(\mu_{B}^{\eta^+}(\sigma_w),\mu_{B}^{\eta^-}(\sigma_w)\big) \le e^{ - \ell/C_*}\,.
    \end{align}
    The vertex $v$ is called $\bad$ otherwise. 
\end{Def}

In both situations, the $\good$ blocks are those in which the influence of a boundary condition, on a distance at least $(\log R)^2$ interior to the block is indeed exponentially small.

\subsubsection{Dominating bad blocks by a sub-critical percolation process}

The following lemmas show that by taking $R$ large, under their respective assumptions, the probability of a vertex in $\Lambda_n^{(R)}$ being $\bad$ can be made arbitrarily small.

\begin{lemma}\label{lem:small-prob-of-bad}
    Under Assumption~\ref{ass:anticoncentrated-field} and the corresponding Definition~\ref{def:anticoncentrated-good-blocks},  if $R$ is at least a $d$-dependent constant, then for every $v\in \Lambda_n^{(R)}$,
    \begin{align*}
        \mathbb P_h(v \text{ is $\bad$}) \le R^{- 5}\,.
    \end{align*}
\end{lemma}

\begin{proof}
    We first observe that under Assumption~\ref{ass:anticoncentrated-field}, the process $(\omega^\influence_x)_x$ of Definition~\ref{def:anticoncentrated-good-blocks} is stochastically below an independent Bernoulli percolation with parameter $1/(20d)$. 
    
    By standard percolation reasoning (comparison to a branching process with offspring distribution $\text{Bin}(2d,1/(20d))$) for every $w\in \Lambda_n$, and every $r\ge 1$, 
    \begin{align*}
        \mathbb P_h \otimes \mathbb P_\eta (|\cC_w(\omega^\influence)| \ge r) \le e^{ - 2r}\,,
    \end{align*}
    where $\mathbb P_\eta$ is the law of the $\text{Ber}(1/(20d)$ process $(\eta_x)_x$. By Markov's inequality, for every $w\in \Lambda_n$, 
    \begin{align*}
        \mathbb P_h\big(\{h: \mathbb P_\eta(|\cC_w(\omega^\influence)|\ge r \mid h)\ge e^{ - r}\}\big) \le e^{ - r}\,.
    \end{align*}
    As such, for each $v\in \Lambda_n^{(R)}$, by a union bound, 
    \begin{align*}
        \mathbb P_h \big(h: \forall w\in B_v\,,\,\forall (\log R)^2\le r\le R/8\,,\, \mathbb P_\eta(|\cC_w(\omega^\influence)|\ge r \mid h) \le e^{-r}\big) \ge 1- (2R)^{d} \sum_{r= (\log R)^2}^{R/8} e^{ -r}\,.
    \end{align*}
    In particular, we get that for $R$ larger than a fixed constant, $\mathbb P_h(v \text{ is }\good) \ge 1-R^{- 5}$. 
\end{proof}

\begin{lemma}\label{lem:small-prob-of-bad-SSM}
    If $\SSM(C)$ holds per Definition~\ref{assump:SSM-RFIM0}, with the corresponding Definition~\ref{def:SSM-good-blocks}, there exists $C_*(d,C), R_0 (d,C)$ such that for all $R\ge R_0$, for every $v\in \Lambda_n^{(R)}$, 
    \begin{align*}
        \mathbb P_h(v \text{ is $\bad$}) \le R^{-5}\,.
    \end{align*}
\end{lemma}

\begin{proof}
    For each fixed $w\in B_v$, $(\log R)^2 \le \ell \le r\le R/8$, $B\in B_r(w) + \{-r+1,...,r-1\}^d$ and $z\in \partial B$ in~\eqref{eq:SSM-good-event} at distance $\ell$ from $w$, let $A_{w,B,z}$ be the complement of the event in~\eqref{eq:SSM-good-event} i.e., $$\{\max_{\xi \in \{\pm 1\}^{\partial B\setminus z}} d_{\sTV} (\mu_{B,\xi^+}(\sigma_w),\mu_{B,\xi^-}(\sigma_w)) > e^{ - \ell/C_*}\}\,.$$ Then, the assumption of $\SSM(C)$ implies by Markov's inequality that  
    \begin{align*}
        \mathbb P(A_{w,B,z}) \le e^{ - \frac{\ell}{C} + \frac{\ell}{C_*}}\,,
    \end{align*}
    so that if $C_* = 2C$, then this is at most $e^{ - \ell/C_*}$. Noting that 
    \begin{align*}
        \{v \text{ is $\bad$}\} \subset \bigcup_{w\in B_v} \bigcup_{\ell,r=(\log R)^2}^{R/8} \bigcup_{B} \bigcup_{\substack{z\in \partial B \\ d(z,w) = \ell}} A_{w,B,z}\,,
    \end{align*}
    a union bound implies 
    \begin{align*}
        \mathbb P(v\text{ is $\bad$}) \le R^d \cdot (R/8) \cdot \sum_{\ell \ge (\log R)^2} (R/8)^d \ell^{d-1} e^{ - \ell/C_*}\,.
    \end{align*}
    Evidently, as long as $R$ is a sufficiently large constant depending on $C_*$ and $d$, the right-hand side is at most $R^{-5}$ 
 (in fact, super-polynomially small) as claimed.  
\end{proof}

   For a realization of $h$, inducing a configuration $\omega^h$, let $\good(h)$ be the subset of $\Lambda_n^{(R)}$ that are $\good$, and let $\bad(h)$ be the subset of $\Lambda_n^{(R)}$ that are $\bad$. Note that under Assumption~\ref{ass:anticoncentrated-field}, $\good(h)$ is an $|h|$-measurable process on $\Lambda_n^{(R)}$ and under $\SSM(C)$, it is an $h$-measurable process on $\Lambda_n^{(R)}$. The states of different $v,w\in \Lambda_n^{(R)}$ are not quite independent, but the following lemma says that in either case, the $\bad$ sites of $\Lambda_n^{(R)}$ are stochastically below a sub-critical independent percolation on $\Lambda_n^{(R)}$.

\begin{lemma}\label{lem:domination-by-independent-percolation}
For every $c>0$, under Assumption~\ref{ass:anticoncentrated-field}, there exists $R_0(c,d)$ such that for all $R>R_0(c)$, under $\mathbb P_h$ 
the set $\bad(h)\subset \Lambda_n^{(R)}$ is stochastically below an independent Bernoulli-$c$ percolation on $\Lambda_n^{(R)}$. 

Under Assumption~\ref{assump:SSM-RFIM0}, there exists $R_0(C,c,d)$ such that the same conclusion holds. 
\end{lemma}

\begin{proof}
    Under Assumption~\ref{ass:anticoncentrated-field}, for every $v\in \Lambda_n^{(R)}$, its status as $\good$ versus $\bad$ is measurable with respect to $(|h_x|)_{x: d_\infty(x,v)\le 5R/4}$, whence the percolation generated by the $\bad$ sites on $\Lambda_n^{(R)}$ is $2$-dependent, i.e., it is independent of the status of all sites at $\Lambda_n^{(R)}$-distance at least $3$ (here we are endowing $\Lambda_n^{(R)}$ with adjacency if vertices are at $\mathbb Z^d$-distance $R$). By \cite{LSS}, for every $c>0$, there exists $c'(c,\Delta)>0$ such that any $2$-dependent percolation on a graph of maximum degree $\Delta$ of parameter $c'$ is stochastically below and independent $c$-percolation. 
    Therefore, by Lemma~\ref{lem:small-prob-of-bad}, if $R$ is sufficiently large, the probability of a vertex being $\bad$ is at most $c'$, whence the set of $\bad$ vertices are stochastically below an independent $c$-percolation. 

    Under Assumption~\ref{assump:SSM-RFIM0}, using Lemma~\ref{lem:small-prob-of-bad-SSM}, the reasoning is identical except that it is now $(h_x)_{x:d_\infty(x,v)\le 5R/4}$ measurable, but at the same time $(h_x)_x$ is an independent process rather than just $(|h_x|)_x$. 
\end{proof}

    To conclude the discussion of the coarse-graining, let us introduce a few notions for the geometry of the $R$-coarse-grained lattice.

   \begin{Def}
     We say two sites in $\Lambda_n^{(R)}$ are $R$-adjacent if they are at $\ell^2$-distance $R$ from one another. We say they are $R$-$*$-adjacent if they are at $\ell^\infty$-distance $R$ from one another. Using these notions of $R$-adjacency, for a subset of $\Lambda_n^{(R)}$, we can define its $R$-clusters (maximal connected components via $R$-adjacency), and analogously $R$-$*$-clusters. The important thing about these notions of adjacency is that $R$-adjacency is ``dual to" $R$-$*$-adjacency, in the sense that if $A$ is not $R$-$*$-connected to $B$, then there is an $R$-cluster separating $A$ from~$B$ and vice versa. 
\end{Def}

   Let $\cC^{(R)}_{w,\bad}(\omega^h) \subset \Lambda_n^{(R)}$ denote the $R$-$*$-connected component of $\bad(h)$ containing $w$. Notice that its $R$-$*$-outer-boundary will form an $R$-connected subset of $\good(h)$ surrounding $w$. Since the degree of $\Lambda_n^{(R)}$ under $R$-$*$-adjacency is at most $3^d$, by Lemma~\ref{lem:domination-by-independent-percolation} with the choice of $c=1/4^d$, say, and standard percolation reasoning, we obtain the following. 

\begin{Cor}\label{cor:largest-bad-cluster}
    Under Assumption~\ref{ass:anticoncentrated-field}, there exists $\kappa(d), R_0(d)$ such that if $R\ge R_0(d)$ and $L\ge \kappa \log n$,  
    \begin{align*}
        \mathbb P_h \Big(\max_{w \in \Lambda_n^{(R)}} |\cC^{(R)}_{w,\bad}(\omega^h)| \ge L\Big) \le e^{-L}\,.
    \end{align*}
    The same holds under Assumption~\ref{assump:SSM-RFIM0} of $\SSM(C)$ with the constants depending on $C$ also.
\end{Cor}

\subsection{Blocks for the block dynamics}
We now define our candidate blocks for the block dynamics using the coarse graining of the previous section. The choice of the blocks will itself  be random, dependent on the realization $h$ through the realization of $\good(h)$ and $\bad(h)$. The crux of the construction is the fact that for any choice of $h$ such that $\max_{w} |\cC^{(R)}_{w,\bad}(\omega^h)|\le L$, the resulting choice of blocks satisfies:
\begin{enumerate}
    \item The volume of every block is at most $L^{1+o(1)}$, and in turn the mixing time of Glauber dynamics on the block is sub-exponential in $L$ by isoperimetry properties in $\mathbb Z^d$.  
    \item The expected Hamming distance of an update on any block with two boundary conditions that differ only at one site is at most $O((\log R)^2)$. 
    \item For every site $v\in \Lambda_n$, the number of blocks for which $v$ is on their boundary divided by the number for which $v$ is in their interior is at most $O(1/R)$. 
\end{enumerate}
In the next subsection, these three properties will be used in conjunction to obtain the claimed mixing time bound of Theorem~\ref{thm:fast-mixing-large-variance}.  

\begin{Def}\label{def:block-dynamics-blocks}
    Fix a realization of $\good(h)$. The blocks $\cB = (\mathcal B_i)_i$ are constructed as follows. 
    \begin{itemize}
        \item \emph{Type 1 blocks}: For each vertex $w\in \Lambda_n$ at distance at most $3R/4$ from an element of $\good(h)$, include the $\ell^\infty$-ball of radius $R/8$ in $\Lambda_n$ centered at $w$ as a block in $\cB$. 
        \item \emph{Type 2 blocks}: For each distinct component $\cC$ in $\{\cC_{w,\bad}^{(R)}(\omega^h)\}_{w\in \Lambda_n^{(R)}}$, let $B_\cC$ consist of all vertices at $\ell^\infty$-distance at most $R$ from $\cC$ and $\ell^\infty$-distance at least $R/2$ from $\good(h)$. Include all translates $(B_\cC + a)_{a\in \{-\frac{R}{8},...,\frac{R}{8}\}^d}$ (intersected with $\Lambda_n$) as blocks in $\cB$. 
    \end{itemize}
    See Figure~\ref{fig:Type-1-2-tiles} for a schematic of the two types of blocks.
\end{Def}

\begin{Obs}\label{obs:type-2-blocks}
    By construction, each $B_\cC$ is exactly a union of cubes of side-length $R$, one centered at each vertex $v\in \cC$. From this we deduce that 
    type 2 blocks originating from some $\cC$ and type 2 blocks originating from some other $\cC'$ must be at least distance  $3R/4$ apart. 
\end{Obs}

\begin{figure}
    \centering
    \includegraphics[width=0.66\textwidth]{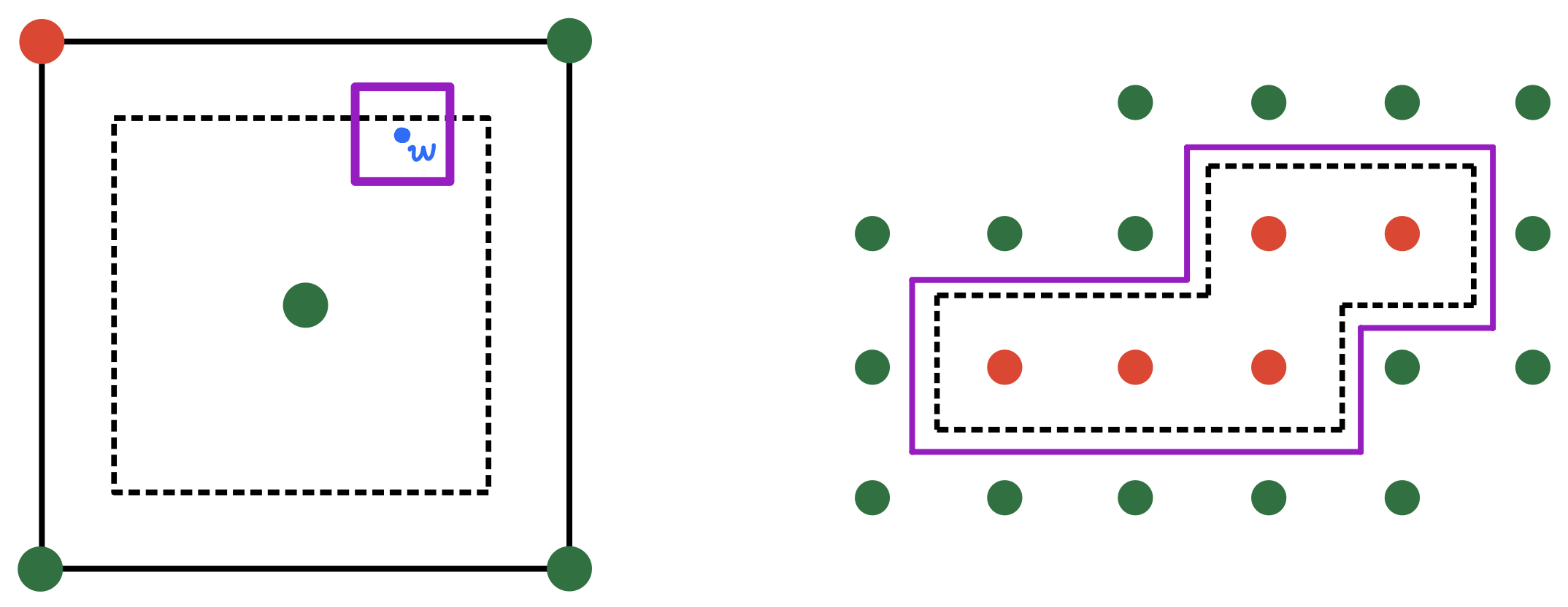}
    \caption{Left: a type 1 block (purple) centered at a vertex $w$ which is close to some coarse-grained vertex in $\good(h)$ (green). Right: a type 2 block (purple) enlarging $B_{\mathcal C}$ (dashed black) for an $R$-$*$-connected component $\mathcal C$ of the coarse-grained set $\bad(h)$ (red).}
    \label{fig:Type-1-2-tiles}
    \vspace{-0.2in}
\end{figure}

We now focus on verifying that the desired properties (1)--(3) hold for this family of blocks. 

\medskip
\noindent \emph{Property (1)}. The volume of a block clearly is at most a factor of $R$ larger than the largest volume of the bad clusters $\cC\in (\cC_{w,\bad}^{(R)}(\omega^h))_{w\in \Lambda_n^{(R)}}$. If we allowed a worst-case mixing time of exponential in the volume per block, the fact that we  need to choose $R$ to be diverging as $\Omega(\log L)$ means the mixing time of the worst block would be super-exponential in $L$, which translates to super-polynomial in $n$. We therefore wish to leverage the classical fact (dating back to canonical path arguments of~\cite{JS}) that the mixing time of a spin system on a graph is bounded by exponential in the cut-width of the graph, while isoperimetry in $\mathbb Z^d$ guarantees that the cut-width of any domain is sub-linear in its volume.  

We appeal to the following mixing time bound for Ising Glauber dynamics on general subsets of $\mathbb Z^d$ using this isoperimetric property. This was proved in~\cite[Theorem 4.12]{CMM-Disordered-magnets} in the absense of an external field, it is easy to verify that all steps in the proof go through mutatis mutandis in the presence of an arbitrary external field. 

\begin{lemma}\label{lem:surface-mixing-time-bound}
        For every $d\ge 2$,  $\beta>0$ there exists $C(d,\beta)>0$ such that the Ising dynamics on a finite connected set $A \subset \mathbb Z^d$ at inverse temperature $\beta$ and arbitrary fields $(h_x)_{x\in A}$ has  
        \begin{align*}
            \min_{\phi}\gap_A^\phi \ge C^{-1} \exp( -C |A|^{\frac{d-1}{d}})\,.
        \end{align*}
        where $\gap_A^\phi$ is the spectral gap of the Ising dynamics on $A$ with boundary conditions $\phi$. 
\end{lemma}

Thus, as long as the largest connected component of $\bad$ blocks in the coarse-graining has volume at most $L$, we deduce a bound on its inverse-gap as sub-exponential in~$L$. This is immediate from Lemma~\ref{lem:surface-mixing-time-bound} together with each block in $\cB$ having volume at most $R^d L$. 

\begin{Cor}\label{cor:single-block-mixing}
    If $h$ is such that $\max_{w\in \Lambda_n^{(R)}} |\cC_{w,\bad}^{(R)}(\omega^h)|\le L$,
    then if $\cB= (\cB_i)_{i}$ are the blocks constructed out of $\good(h)$ in Definition~\ref{def:block-dynamics-blocks}, then 
    \begin{align*}
        \min_{i}  \min_\phi\gap_{\cB_i}^{\phi}\ge C^{-1} \exp( - C R^{d-1} L^{\frac{d-1}{d}}) \,.
    \end{align*}
\end{Cor}

\medskip
\noindent \emph{Property (2).}
We now show using the definition of the coarse-graining and the blocks, that they ensure a block update doesn't increase the discrepancy between two configurations too much. 

\begin{lemma}\label{lem:expected-hamming-distance}
    Fix any realization of $h$, and any block $\cB_i$ constructed out of $\good(h)$ per Definition~\ref{def:block-dynamics-blocks}. 
   For any boundary condition $\xi$ let $\xi^y$ be the boundary condition that agrees with $\xi$ everywhere except at $y$, where it takes the opposite spin. For every $y\in \partial \cB_i$, and $(\xi,\xi^y)$, there exists a coupling $\mathbf{P}$ of $\sigma\sim \mu_{\cB_i}^\xi$ and $\sigma^y\sim \mu_{\cB_i}^{\xi^y}$ such that 
   \begin{align*}
       \mathbf{E}[d_H(\sigma,\sigma^y)] \le  (\log R)^2 + |\cB_i| e^{ - R/8} + O(1)\,,
   \end{align*}
   where $d_H$ denotes Hamming distance. 
\end{lemma}

We have to prove Lemma~\ref{lem:expected-hamming-distance} separately depending on the definition of $\good$ blocks we used (i.e., depending on which of Assumption~\ref{ass:anticoncentrated-field} and Assumption~\ref{assump:SSM-RFIM0} we work under). 

\begin{proof}
    If $\cB_i$ is a type (1) block per Definition~\ref{def:block-dynamics-blocks}, then every vertex on its boundary is at distance at most $7R/8$ from an element of $\good(h)$. If it is a type (2) block, every boundary vertex of $B_\cC$ is exactly distance $R/2$ from $\good(h)$, so any of its translates by $R/8$ must be within distance $5R/8$ from $\good(h)$. In either case, $y$ must have distance strictly less than $7R/8$ from $\good(h)$.

    \medskip
    \noindent    \emph{Proof under $\good$-blocks of Definition~\ref{def:anticoncentrated-good-blocks}}. The above implies by definition of $\good$, that 
\begin{align}\label{eq:boundary-of-block-good}
    \mathbb P_\eta(|\cC_y(\omega^\influence)| \ge r \mid h) \le e^{ - r} \quad \text{ for all }(\log R)^2 \le r\le R/8\,.
\end{align}  
We use the random-variables $\rho$ to construct a coupling of $\sigma\sim \mu_{\cB_i}^\xi$ and $\sigma^y\sim\mu_{\cB_i}^{\xi^y}$ as follows. 
\begin{enumerate}
    \item Assign every vertex an i.i.d. Unif$[0,1]$ random variable $U_x$.
    \item One at a time starting from neighbors of $y$, use the uniforms to generate $\sigma,\sigma^y$ in a coupled manner, where for each vertex $x$ having $|h_x|\ge K$, if $U_x \le 1-\rho$ (for $\rho$ from~\eqref{eq:rho-parameter}) necessarily both $\sigma_x$ and $\sigma_x'$ take the value $\sgn(h_x)$. 
    \item As soon as we have encountered a set of vertices separating $y$ from the rest of $\cB_i$ all of which take the value of $\sgn(h_x)$ in both $\sigma,\sigma^y$ use the identity coupling on the remainder of the configuration (since it will have the same induced boundary conditions on it in both $\sigma,\sigma^y$). 
\end{enumerate}
This is easily checked to be a valid coupling. Moreover, generating the $\eta_x$ (used to construct $\omega^\influence$) via $\mathbf{1}_{U_x \ge 1-\rho}$, the boundary of $\cC_y(\omega^\influence)$ always serves as a set on which $\sigma$ and $\sigma^y$ both take the value of $\sgn(h_x)$, so in particular, this coupling satisfies 
\begin{align*}
    \sigma(\cB_i \setminus \cC_y(\omega^\influence))  = \sigma^y(\cB_i \setminus \cC_y(\omega^\influence))\,.
\end{align*}
Therefore, we have 
\begin{align*}
    \mathbf E[d_H(\sigma,\sigma')] \le (\log R)^2 + \sum_{r = (\log R)^2}^{R/8} r \mathbb P_\eta(|\cC_w(\omega^\influence)|\ge r \mid h) + |\cB_i| \mathbb P_\eta(|\cC_w(\omega^\influence)| \ge R/8 \mid h)\,.
\end{align*}
Plugging in the bound of \eqref{eq:boundary-of-block-good} on the probabilities above, we evidently get the desired bound up to change of the constant $C$. 

\medskip
\noindent \emph{Proof under $\good$-blocks of Definition~\ref{def:SSM-good-blocks}}. 
    As reasoned at the beginning of the proof, by construction $y$ is at distance at most $7R/8$ from $\good(h)$. 
    We now construct a monotone coupling $\mathbf{P}$ of $\sigma\sim \mu_{\cB_i}^\xi$ and $\sigma^y\sim \mu_{\cB_i}^{\xi^y}$ as follows. 
    \begin{enumerate}
        \item Assign every vertex an i.i.d. Unif$[0,1]$ random variable $U_x$
        \item Iteratively for $\ell \ge 1$, expose the spins $(\sigma_x,\sigma^y_x)_{x\in \cB_i: d_1(x,y)=\ell}$ coupled via the monotone coupling using the common random variables $U_x$. 
    \end{enumerate}
    Notice that under this coupling, if $\sigma_x = \sigma^y_x$ for all $x\in \cB_i: d_1(x,y) = \ell$, then $\sigma$ and $\sigma^y$ agree on all vertices at distance greater than $\ell$ from $y$. This is because on that event, the induced boundary conditions on the un-revealed spins of $\cB_i$ will agree completely. In particular, in that case, $d_H(\sigma,\sigma^y)$ would be at most $\ell^d$. Furthermore, we always have the deterministic bound $d_H(\sigma,\sigma^y)\le |\cB_i|$. Therefore, we can bound 
    \begin{align}\label{eq:expected-hamming-distance}
        \mathbf E[d_H(\sigma,\sigma^y)] & \le (\log R)^{2d} +  \sum_{\ell = (\log R)^2}^{R/8} \mathbf{P}(d_H(\sigma,\sigma^y) \ge \ell^d) + |\cB_i| \mathbf{P}(d_H(\sigma,\sigma^y)\ge (R/8)^d)\,.
    \end{align}
    By the above reasoning, the tail probabilities for $(\log R)^2 \le \ell \le R/8$ are bounded as 
    \begin{align}\label{eq:Hamming-distance-tail}
           \mathbf{P}(d_H(\sigma,\sigma^y)\ge \ell^d)\le  \ell^{d-1} \max_{x\in \cB_i: d_1(x,y) = \ell} \mathbf{P} (\sigma_x \ne \sigma^y_x)\,.
    \end{align}
    Since the coupling $\mathbf{P}$ is monotone, the disagreement probability on the right is exactly 
    \begin{align}\label{eq:coupling-prob-tv-dist}
        \mathbf{P}(\sigma_x \ne \sigma^y_x) = d_{\sTV}(\mu_{\cB_i}^{\xi}(\sigma_x\in \cdot), \mu_{\cB_i}^{\xi^y}(\sigma_x^y\in \cdot))\,.
    \end{align}
    For that tv-distance, inscribe a ball $B$ of $\ell^\infty$-radius $\ell$, say, in $\cB_i$, such that all points in $\partial B$ at graph distance less than $\ell$ from $x$ belong to $\partial \cB_i$. This ensures that the boundary conditions induced by $\mu_{\cB_i}^\xi,\mu_{\cB_i}^{\xi^y}$ on the vertices in $\partial B$ at distance less than $\ell$ from $x$ are identical. (This inscription is always possible since by Observation~\ref{obs:type-2-blocks}, all $\cB_i$ can be partitioned into disjoint cubes of side-length $R/2$, and $\ell \le R/8$.) The vertices in $\partial B$ at distance at least $\ell$ from $x$ may all differ; call two such boundary conditions $\zeta,\zeta'$ on $\partial B$. By flipping one spin at a time to interpolate between $\zeta, \zeta'$, by~\eqref{eq:SSM-good-event} (which applies since $x$ is at distance at most $R/8$ from $y$ and thus at distance at most $R$ from a good vertex $v\in \good(h)$),
    \begin{align*}
        d_{\sTV}(\mu_{\cB_i}^\xi(\sigma_x),\mu_{\cB_i}^{\xi^y}(\sigma_x))\le \sum_{z\in \partial B: d(z,x)\ge \ell} \max_{\zeta\in \{\pm 1\}^{\partial B\setminus z}} d_{\sTV}(\mu_{B}^{\zeta^+}(\sigma_x),\mu_{B}^{\zeta^-}(\sigma_x)) \le \ell^{d-1} e^{ - \ell/C_*}\,.
    \end{align*}
    Altogether, plugging this into~\eqref{eq:coupling-prob-tv-dist}, then~\eqref{eq:Hamming-distance-tail}, then~\eqref{eq:expected-hamming-distance}, we get the bound
    \begin{align*}
        \mathbf{E}[d_H(\sigma,\sigma')] \le (\log R)^{2d} + |\cB_i| e^{ - R/16C_*} + O(1)\,,
    \end{align*}
    as desired.
\end{proof}

\medskip
\noindent \emph{Property (3).}
The last property we check is that no vertex is placed on the boundary of too many blocks compared to the number to which it is interior. 

\begin{lemma}\label{lem:vertices-boundary-of-blocks-vs-interior}
    For any realization of $h$, and $\cB = (\cB_i)$ constructed per Definition~\ref{def:block-dynamics-blocks}, we have 
    \begin{align*}
        \max_{w\in \Lambda_n} \frac{|\{i: w\in \partial \cB_i\}|}{|\{i: w\in \cB_i\}|} \le \frac{32 d 2^d}{R}\,.
    \end{align*}
\end{lemma}
\begin{proof}
    Fix any vertex $w\in \Lambda_n$ We consider different cases depending on the distance of $w$ to $\good(h)$ and $\bad(h)$. 
    
    If $w$ is at distance at most $5R/8$ from an element of $\good(h)$, then the $\ell^\infty$-balls of radius $R/8$ centered at every vertex distance less than $R/8$ are included in $\cB$, and they all have $w$ in their interior, so that $$|\{i: w\in \cB_i\}|\ge (R/4)^d\,.$$ 
    The number of type-1 blocks having $w$ on their boundary is at most $(2d) (R/4)^{d-1}$ for those centered at vertices exactly at distance $R/8$ from $w$. By Observation~\ref{obs:type-2-blocks}, all type-2 blocks having $w$ on their boundary must be originating from the same $\cC$. Moreover, at most $2d (R/4)^{d-1}$ of the translates of $B_\cC$ can have $w$ on their boundary since $B_\cC$'s boundary is a subset of the boundaries of its constituent $R\times \cdots \times R$ cubes from the second part of Observation~\ref{obs:type-2-blocks}. Combining the above, 
    \begin{align*}
        \frac{|\{i: w\in \partial \cB_i\}|}{|\{i: w\in \cB_i\}|} \le \frac{32d}{R} \quad \text{ if $d_\infty(w,\good(h)) \le 5R/8$}
    \end{align*}

    If $w$ is at distance at least $5R/8$ from $\good(h)$, clearly it must also have distance at most $3R/8$ from an element of $\bad(h)$. If $\cC$ is the cluster of bad sites of the coarse graining it is closest to, then $w$ must be interior to all $(R/4)^{d}$ translates of $B_\cC$. On the other hand, it can be on the boundary only of type-1 blocks centered at vertices exactly distance $R/8$ away, there being at most $(2d)(R/4)^{d-1}$ many of these, that leaves 
    \begin{align*}
        \frac{|\{i: w\in \partial \cB_i\}|}{|\{i: w\in \cB_i\}|} \le \frac{16d}{R} \quad \text{ if $d_\infty(w,\good(h)) \ge 5R/8$}\,.
    \end{align*}
    Together, these bounds give the claim for vertices $w$ at distance at least $R$ from the boundary of $\Lambda_n$ when viewed as a subset of $\mathbb Z^d$. 

    If $w$ at distance at most $R$ from the boundary of $\Lambda_n$, there may be as few as a $1/2^d$ fraction of the number of blocks containing $w$ in their interiors as for $w$ far from the boundary (e.g., if $w$ is a corner of $\Lambda_n$). The number of blocks containing $w$ on their boundary is upper bounded by the same number as when $w$ is far from the boundary. Thus, multiplying the earlier bound of $32 d/R$ by $2^d$, we get the worst-possible ratio one could arrive at. 
\end{proof}

\subsection{From block dynamics to a mixing time bound}

We now use the blocks from the previous section to deduce the claimed bound on the mixing time of the $\RFIM$ on $\Lambda_n$ with large variance. Given the properties (1)--(3) proved in the previous subsection, stitching them together is an application of the following block dynamics tool in Markov chain mixing times, which follows from decomposition of the variational form of the spectral gap, and which has found many uses in bounding spin system mixing times.

\begin{Def}\label{def:block-dynamics}
    If blocks $(\cB_i)$ cover $V$, the \emph{block dynamics} with blocks $\cB$ is the continuous-time Markov chain that assigns each block a rate-1 Poisson clock, and when the clock corresponding to block $\cB_i$ rings, it fully resamples the configuration on $\cB_i$ conditional on the configuration on~$V \setminus \cB_i$. 
\end{Def}

\begin{Prop}[{\cite[Proposition 3.4]{MartinelliLectureNotes}}]\label{prop:block-dynamics}
    Consider a graph $G = (V,E)$, and blocks $\cB_1,...,\cB_k \subset V$ such that $\bigcup_i \cB_i = V$. 
    Let $\gap_G$  and $\gap_{\cB_i}^\phi$ be the spectral gaps of 
continuous-time single-site dynamics on $G$ and on $\cB_i$ with boundary condition $\phi$. Let $\gap_{\cB}$ be the spectral gap of the corresponding block dynamics, per Definition~\ref{def:block-dynamics}. Letting $\chi = \sup_{v\in V} \#\{i:  v \in \cB_i\}$, we have
\begin{align*}
    \gap_{G} \ge \chi^{-1} \gap_{\cB}\inf_{i}\inf_{\phi} \gap_{\cB_i}^\phi\,.
\end{align*}
where the infimum over $\phi$ is over boundary conditions on $\cB_i$, i.e., assignments to $V \setminus \cB_i$. 
\end{Prop}

The proof of Theorem~\ref{thm:fast-mixing-large-variance} essentially breaks into the following two lemmas controlling the inverse gap of the block dynamics and the individual blocks (the quantity $\chi$ is easily bounded). 

\begin{lemma}\label{lem:block-dynamics-gap}
    Fix $\good(h)$ and let $\cB_i$ be constructed per Definition~\ref{def:block-dynamics-blocks}. Then, if we have that $R \ge C_0 \log (\max_i |\cB_i|)$ for a large enough constant $C_0$, there is a constant $C'$ such that 
    \begin{align*}
        \gap_{\cB}^{-1} \le C'\,.
    \end{align*}
\end{lemma}
\begin{proof}
    The proof goes by path coupling. Let us consider the discrete time version of the block dynamics where a block is chosen uniformly at random at each discrete time step. The inverse gap of the discrete and continuous time chains are within a factor of the number of blocks of one another, which is evidently on the order of $|\Lambda_n| R^d$. 

    Consider the Hamming distance on configurations. Recalling the path coupling from \cite[Theorem 14.6]{LP}, for each pair of adjacent configurations $\sigma,\sigma^y$ that differ at vertex $y$, we wish to construct a coupling $\mathbf{P}$ of their respective 1-step transitions $X_1 \sim P_{\cB} \sigma$ and $Y_1 \sim P_\cB \sigma^y$ such that 
    \begin{align*}
        \mathbf{E}[d_H(X_1,Y_1)]\le \theta<1\,.
    \end{align*}
    If this holds for a $\theta<1$ for all pairs $(\sigma,\sigma^y)$ then 
    \begin{align*}
        \mathbf{E}[d_H(X_t,Y_t)]\le \theta^t |\Lambda_n|\qquad \text{and} \qquad \gap_{\cB}^{-1} \le (1-\theta)^{-1}\,.
    \end{align*}
    We prove this with $\theta = 1- \Theta(\frac{1}{|\Lambda_n|})$ for the discrete-time block dynamics, which implies the claimed bound for the continuous-time block dynamics. 

    Fix any $y\in \Lambda_n$, and let $\sigma^+,\sigma^-$ be two arbitrary configurations that agree on all sites except at $y$ where they take values $+,-$ respectively. Let $\cB_i$ be the next block whose clock rings (this is uniformly chosen amongst the blocks, and we use the same choice for the coupling of $X_1,Y_1$). On the event that $y\notin \cB_i \cup\partial \cB_i$, the expected Hamming distance between $X_1,Y_1$ is still $1$. On the event that $y\in \cB_i$, the identity coupling is used, the boundary conditions induced by $\sigma^+,\sigma^-$ on $\cB_i$ are identical, and the expected Hamming distance is $0$. Finally, on the event that $y\in \partial \cB_i$, the coupling of Lemma~\ref{lem:expected-hamming-distance} is used, whence the expected Hamming distance is at most $2(\log R)^2$ as long as $R$ is at least $C_0 \log |\cB_i|$. Putting these together, we see that 
    \begin{align*}
        \mathbf{E}[d_H(X_1,Y_1)] \le 1 - \frac{1}{|\cB|} \Big( |\{i: y\in \cB_i\}| - 2 |\{i: y\in \partial \cB_i\}| (\log R)^2\Big)\,.
    \end{align*}
    Noting that $|\cB|\le |\Lambda_n|R^d$ and $|\{i: y\in \cB_i\}| \ge (R/4)^d$, and then using Lemma~\ref{lem:vertices-boundary-of-blocks-vs-interior}, this gives 
    \begin{align*}
         \mathbf{E}[d_H(X_1,Y_1)] & \le 1- \frac{1}{4^d |\Lambda_n|} \Big( 1 - \frac{|\{i: y\in \partial \cB_i\}|}{|\{i:y\in \cB_i\}|} 2(\log R)^2\Big) \\ 
         & \le 1- \frac{1}{4^d |\Lambda_n|} \Big( 1- \frac{64 d 2^d (\log R)^2}{R}\Big)\,.
    \end{align*}
    At this point, we find that as long as $R$ is sufficiently large as a function of $d$, we have contraction with $\theta = 1- \Theta(\frac{1}{|\Lambda_n|})$ concluding the proof. 
\end{proof}

\begin{proof}[Proof of Theorem~\ref{thm:fast-mixing-large-variance}]
    Fix a realization of $\good(h)$ and let $\cB= (\cB_i)_i$ be the corresponding set of blocks defined per Definition~\ref{def:block-dynamics-blocks}. Suppose that $h$ is such that $\max_{w \in \Lambda_n^{(R)}} |\cC_{w,\bad}^{(R)}(\omega^h)| \le L$, (the complement of this event having $\mathbb P_h$-probability $e^{-L}$ by Corollary~\ref{cor:largest-bad-cluster}). By construction, this implies that $\max_i |\cB_i| \le R^d L$. 
    We now make the concrete choice 
    \begin{align*}
        R = C_1 \log L\,,
    \end{align*}
    for a sufficiently large constant $C_1$.  
    The quantity $\chi$, i.e., the maximum number of blocks confining a single vertex is easily seen to be at most $2 (R/4)^d \le O((\log L)^d)$. 

    Since $\max_i |\cB_i| \le L R^d$ for large $n$, if $C_1$ is sufficiently large then $R \ge C_0\log (\max_i |\cB_i|)$  so we can apply Lemma~\ref{lem:block-dynamics-gap} whence $\gap_{\cB}^{-1} \le C'$ for some $C'$. 
    
    At the same time, by Lemma~\ref{cor:single-block-mixing}, we get that 
    \begin{align*}
        (\min_i \min_{\phi} \gap_{\cB_i}^\phi)^{-1} \le C \exp(C R^d L^{\frac{d-1}{d}})\,.
    \end{align*}
     Combining these three bounds into Proposition~\ref{prop:block-dynamics}, we deduce that 
    \begin{align*}
        \gap_{\Lambda_n}^{-1} \le \exp \big(C L^{\frac{d-1}{d}} (\log L)^{d-1}\big)\,,
    \end{align*}
    concluding the proof of the theorem. 
\end{proof}

\subsection{SSM holds if the field is sufficiently anti-concentrated}
\label{subsec:large-variance-implies-ssm}
In this short subsection, we prove that if $|h|$ is i.i.d.\ and the variance of $|h|$ is sufficiently large, then $\SSM(C)$ holds. Note: this isn't enough to just reduce the proof with Assumption~\ref{ass:anticoncentrated-field} to the proof with Assumption~\ref{assump:SSM-RFIM0} because the proof with the $\SSM$ assumption requires independence of $(h_x)_x$, which the proof with Assumption~\ref{ass:anticoncentrated-field} does not require, and the reduction from $\WSM$ after stochastic localization does not have. 

\begin{lemma}\label{lem:ssm_large_fields}
    Suppose that Assumption~\ref{ass:anticoncentrated-field} holds. Then there is $C$ such that the $\RFIM$ has $\SSM(C)$.  
\end{lemma}
\begin{proof}
    Fix a domain $B$ containing the origin $o$, and fix a $z\in \partial B$. We consider the process $(\eta_x)$ which is $1$ if $|h_x| \le K$. As before, $K$ is chosen such that $\rho(K)$ from~\eqref{eq:rho-parameter} is at most $1/10d$, say. In that case, wherever $\eta_x =0$, the probability that $\sigma_x$ agrees with $\text{sign}(h_x)$ is at least $1-1/10d$, uniformly over the neighboring spins to $x$. 

    We construct a coupling of samples $\sigma^+ \sim \mu_{B}^{\xi^+}$ and $\sigma^-\sim \mu_{B}^{\xi^-}$ disagreeing at $z\in \partial B$, at the origin $\sigma_0$ as follows. Starting from neighbors of $z$, use the monotone coupling (with the same uniform random variables at each vertex) to iteratively reveal the spins of $\sigma^+$ and $\sigma^-$, exposing the maximal component of $z$ on which the two disagree. By the above reasoning, the probability that they disagree at a site $x$, even conditionally on the fields and spins at all other sites, is at most $1/5d$. Therefore, the disagreement region is confined by an independent $\text{Ber}(1/5d)$ percolation, and beyond the component of $z$ in this percolation, the two configurations are coupled via the identity coupling. Since the degree of $\mathbb Z^d$ is $2d$, the $\text{Ber}(1/5d)$ percolation is subcritical, and $\max_{\eta \in \{\pm 1\}^{\partial B\setminus \{z\}}} d_{\sTV}(\mu_B^{\xi^+}(\sigma^+_0),\mu_B^{\xi^-}(\sigma^-_0))$ is bounded by the indicator that this sub-critical percolation connects $z$ to $0$ in $B$. This event has probability at most $e^{ - d(0,z)/C}$ for some $C(d)$, so that also bounds the expectation of its indicator, yielding the claim.  
\end{proof}

\section{Construction of a super-polylogarithmic lower bound}
\label{sec:lower-bound}

Given the fact that our mixing time upper bound, even at very large variance is $n^{o(1)}$, one may wonder what the true mixing time should be. Standard arguments (see~\cite{Hayes-Sinclair}) give a coupon collecting lower bound of $O(\log n)$, which also applies in the absence of an external field. In what follows we show that when the temperature is below the zero-field critical point, the mixing time of the $\RFIM$ is at least super-polylogarithmically slower. 

To be precise, we construct an  $\exp((\log n)^{c})$ lower bound for the mixing time of the $\RFIM$ with an $O(1)$ variance, when the temperature $\beta$ is large. We will phrase this for Gaussian external fields though a distribution with density at zero, for instance, would suffice.  

\begin{Prop}\label{prop:mixing-lower-bound}
    For $\beta$ sufficiently large as a function of $d$, and i.i.d.\ Gaussian external fields $(h_x)_x$, the mixing time of the continuous-time Glauber dynamics for the $\RFIM$ on $\In$ with $\var(h_x) = \Omega(1)$ is at least $\exp((\log n)^{(d-1)/4d})$. 
\end{Prop} 

The lower bound in fact in all $(d,\beta)$ regimes where there exists a (sequence of) boundary conditions $\zeta_L$ under which the mixing time of the zero-field Ising model on boxes of side-length $L$ with boundary conditions $\zeta_L$ is  $\exp(\Omega(L^{d-1}))$. We give concrete examples below when $d=2$ for all $\beta>\beta_c(d)$, and when $d\ge 3$ for $\beta$ sufficiently large depending on $d$, but we expect that even when $d\ge 3$ such a boundary condition exists for all $\beta>\beta_c(d)$. 

\begin{Ex}
    \label{ex:slow-bc-low-temp-Ising}
    In $d=2$ and any $\beta>\beta_c$, if $\zeta_L$ is $+$ on the top and bottom of $\Lambda_L$ and $-$ on the left and right, it induces $\exp(L^{d-1})$ slow mixing (see e.g.,~\cite[Theorem 1 together with Remark 1.2]{GL16c}). In $d\ge 3$ and $\beta>\beta_0$ sufficiently large, if $\zeta_L$ is all $-$, except a strip of $+$ of width $\epsilon n$, say $\{x\in \partial \Lambda_L: x_d\in [\frac{n}{2},\frac{n}{2}+\epsilon n]\}$, it induces slow mixing. The slowdown in this latter case follows from the rigidity of the $+/-$ Dobrushin interface~\cite{Dobrushin72a}.  
\end{Ex}

We expect the true order of the inverse gap and mixing time to be $\exp((\log n)^{\frac{d-1}{d} +o(1)})$ so the factor $4$ in the denominator of the bound of Proposition~\ref{prop:mixing-lower-bound} is likely sub-optimal.

\begin{proof}[Proof of Proposition~\ref{prop:mixing-lower-bound}]
Let $m=(\log n)^{1/4d}$. On a cube of side-length $m$, let $\zeta_m$ be the boundary condition assignment of Example~\ref{ex:slow-bc-low-temp-Ising}. The domain $\Lambda_n$ can evidently be tiled by $\lfloor(\frac{n}{m+1})^d\rfloor$ many disjoint cubes of side-length $m$, at distance at least $2$ from one another. Denote these cubes by  $(B_i)_{i}$. We will use the event that one of these has a bad realization of the external field to embed the assumed slow mixing with boundary conditions $\zeta_m$ into $B_i$ for some $i$. For each $i$, let $E_i$ be the following $h$-measurable event:  
\begin{enumerate}
    \item For all $v\in \partial B_i$, the external field $h_v$ has $|h_v|\ge m^d$, and sign given by $\zeta_m$; and 
    \item Every vertex in $B_i$ has external field at most $1/m^d$ in absolute value.
\end{enumerate}
By bounds on the Gaussian distribution function, there exists a constant $c$ depending on $\text{Var}(h_v)$, such that this pair of events has probability $(e^{ - cm^{2d}})^{m^{d-1}}$ for item~(1), and $(c/m^d)^{m^d}$ for item~(2). Since these events are independent, $\mathbb P(E_i)\ge e^{ - m^{3d}} \ge n^{-o(1)}$, whence the independence of $(E_i)_i$ and standard concentration for a binomial random variable implies that 
\begin{align*}
    \mathbb P\Big(\bigcup_{i=1}^{(\frac{n}{m+1})^d} E_i\Big) \ge 1-o(1)\,. 
\end{align*}
We claim that on the event $\bigcup_i E_i$ the mixing time is at least $e^{m^{d-1}}$. Take any such external field realization, and let $B_\star = B_i$ for an (arbitrary) $i$ such that $E_i$ holds. Let $X_{t,\star}$ be the projection of the Glauber dynamics onto $B_\star\cup \partial B_\star$, initialized from $\zeta_m$ on $\partial B_\star$, initialized interior to $B_\star$ from the initialization attaining the zero-field $e^{m^{d-1}}$ mixing time guaranteed to exist per Example~\ref{ex:slow-bc-low-temp-Ising}, and initialized arbitrarily outside of $B_\star \cup\partial B_\star$.

We begin by claiming that the total-variation distance between $X_{t,\star}$ and a Glauber dynamics chain on $B_*$ with frozen $\zeta_m$ boundary conditions is $o(1)$ for all $t\le e^{m^d}$. Indeed, by Poisson tail bounds and a coupon collector bound, with probability $1-o(1)$, there are at most $2 t m^d$ many clock rings on $B_\star \cup \partial B_\star$ by time $t$. In each of these clock rings, the probability that a site on $\partial B_\star$ changes sign is at most $e^{-m^d + 2\beta d}$. Therefore, if $\widetilde X_{t,\star}$ is a Glauber dynamics only making updates on $B_\star$, with boundary conditions $\zeta_m$ on $\partial B_\star$, with $\widetilde X_{0,\star}= X_{0,\star}$, then under the identity coupling of updates inside $B_\star$, with probability $1-o(1)$ we have $\widetilde X_{t,\star} = X_{t,\star}$ for all $t= e^{ o(m^d)}$.

We next claim that the mixing time of $\widetilde X_{t,\star}$ is bounded by $e^{ \Omega(m^{d-1})}$. To see this, observe that the ratio of the stationary measure on $B_\star$ with its external field and without any field, is at most $\exp(\sum_{v\in B_\star} |h_v|) = e^{ O(1)}$ since $|h_v|\le m^{-d}$ for all $v\in B_\star$; furthermore, the transition rates for the Glauber dynamics on $B_\star$ with and without field are within a factor of $e^{m^{-d}} = 1+o(1)$ of one another. Together, the variational form of the spectral gap implies that the spectral gap of $\widetilde X_{t,\star}$ is within a constant factor of the spectral gap of Glauber dynamics on a cube of side-length $m$ with boundary conditions $\zeta_m$ and no external field. By the choice of $\zeta_m$, this inverse gap is $\exp(\Omega(m^{d-1}))$, and in turn the mixing time of $\widetilde X_{t,\star}$ is at least $\exp(\Omega(m^{d-1}))$.


Therefore, there exists an event measurable with respect to the interior spins of $B_\star$ which gets at least $1/4$ more mass under the stationary measure on $B_\star$ than it does under $\widetilde{X}_{t,\star}(B_\star)$ while $t\le e^{cm^{d-1}}$. That implies the same for $X_{t,\star}(B_\star)$ with $1/4$ replaced by $1/4+o(1)$. Extending this event to an event on the full configuration on $\Lambda_n$ by allowing arbitrary configurations exterior to $B_\star \cup \partial B_\star$, this implies the claimed lower bound of $\exp(\Omega(m^{d-1}))$ on the mixing time of $X_{t}$.
\end{proof}

\printbibliography

\end{document}